\documentclass[a4paper,11pt]{article}
\usepackage[ngerman, english]{babel}
\usepackage[T1]{fontenc}
\usepackage[utf8]{inputenc}
\usepackage{amsmath}
\usepackage{amssymb}
\usepackage{amsthm}
\usepackage{graphicx}
\usepackage{here}
\usepackage{color}
\usepackage{xcolor}
\usepackage{enumerate}
\usepackage{lmodern}
\usepackage{fancyvrb}
\usepackage[plainpages=false]{hyperref}
\usepackage{caption}
\usepackage{subfigure}
\usepackage{epstopdf}
\newcommand{\numberset}{\mathbb}

\newcommand{\Z}{\numberset{Z}}
\newcommand{\ca}{\mathcal}

\newtheorem{defi}{Definition}[section]
\newtheorem{theo}[defi]{Theorem}

\newtheorem{lem}[defi]{Lemma}
\newtheorem{con}[defi]{Conjecture}
\newtheorem{obs}[defi]{Observation}
\newtheorem{cor}[defi]{Corollary}
\newtheorem{prob}[defi]{Problem}

\newcounter{claimcount}
\newenvironment{claim}{\refstepcounter{claimcount}\textbf{Claim \arabic{claimcount}.}}{}

\usepackage{etoolbox}
\AtBeginEnvironment{proof}{\setcounter{claimcount}{0}}

\theoremstyle{remark}

\newcommand{\ENDproof}{\hfill $\blacksquare$\medskip\par}

\title{Sets of $r$-graphs that color all $r$-graphs}

\author{Yulai Ma$^1$\thanks{Supported by Sino-German (CSC-DAAD) Postdoc Scholarship Program 2021 (57575640)}, Davide Mattiolo$^2$\thanks{Supported by a Postdoctoral Fellowship of the Research Foundation Flanders (FWO), project number 1268323N}, Eckhard Steffen$^1$, Isaak H.~Wolf$^1$ \thanks{Funded by Deutsche Forschungsgemeinschaft (DFG) - 445863039} \\
\footnotesize
		$^1$ Department of Mathematics, Paderborn University, Warburger Str.\ 100, 33098 Paderborn,
		Germany.
	\\
	\footnotesize
	$^2$ Department of Computer Science, KU Leuven Kulak, 8500 Kortrijk, Belgium.
\\ \footnotesize yulai.ma@upb.de, davide.mattiolo@kuleuven.be, es@upb.de, isaak.wolf@upb.de}

\date{}

\begin{document}

\maketitle

\begin{abstract}
An $r$-regular graph is an $r$-graph, if every odd set of vertices is connected to its complement by at least $r$ edges. Let $G$ and $H$ be $r$-graphs. An \emph{$H$-coloring} of $G$ is a mapping 
$f\colon E(G) \to E(H)$ such that each $r$ adjacent edges of $G$ are mapped to $r$ adjacent edges of $H$. For every $r\geq 3$, let $\ca H_r$ be an inclusion-wise minimal set of connected $r$-graphs, such that for every connected $r$-graph $G$ there is an $H \in \ca H_r$ which colors $G$.

We show that $\ca H_r$ is unique and characterize $\ca H_r$ by showing that 
$G \in \ca H_r$ if and only if the only connected $r$-graph coloring $G$ is $G$ itself. 

The Petersen Coloring Conjecture states that the Petersen graph $P$ colors every bridgeless cubic graph. 
We show  that if true, this is a very exclusive situation. 
Indeed, either $\ca H_3 = \{P\}$ or $\ca H_3$ is an infinite set and if
$r \geq 4$, then $\ca H_r$ is an infinite set. Similar results hold for the restriction on simple $r$-graphs.    

By definition, $r$-graphs of class $1$ (i.e.\ those having edge-chromatic number equal to $r$) can be colored with any $r$-graph.
Hence, our study will focus on those $r$-graphs whose edge-chromatic number is bigger than $r$, also called $r$-graphs of class $2$. We determine the set of smallest $r$-graphs of class 2 and show that it is a subset of $\ca H_r$.

%Let $\ca T_r$ be the set of smallest $r$-graphs which are class 2,  where an $r$-graph is class 2 if its edge-chromatic number is bigger than $r$.
%We determine $\ca T_r$ and show that $\ca T_r$ is a proper subset of $\ca H_r$.  

\end{abstract}

{\bf Keywords:} perfect matchings, regular graphs, factors, $r$-graphs, edge-coloring, class 2 graphs, Petersen Coloring Conjecture, Berge-Fulkerson Conjecture.

\section{Introduction}
All graphs considered in this paper are finite and may have parallel edges 
but no loops.
The vertex set of a graph $G$ is denoted by $V(G)$ and its edge set by $E(G)$. 
A graph is \emph{$r$-regular} if every vertex has degree $r$. An
$r$-regular graph is an \emph{$r$-graph}, if $|\partial_G(X)| \geq r$ for every $X \subseteq V(G)$ of odd cardinality, where $\partial_G(X)$ denotes the set of edges 
that have precisely one vertex in $X$. 

Let $G$ be a graph and $S$ be a set. An \emph{edge-coloring} of $G$ is a mapping $f\colon E(G)\to S$. It is a 
\emph{$k$-edge-coloring} if $|S|=k$, and it is \emph{proper} 
if $f(e) \not = f(e')$ for any two adjacent edges $e$ and $e'$. 
The smallest integer $k$ for which $G$ admits a proper $k$-edge-coloring 
is the \emph{edge-chromatic number} of $G$, which is denoted by $\chi'(G)$. A \emph{matching} is a set $M\subseteq E(G)$ such that no two edges of $M$ are adjacent. Moreover, $M$ is said to be \emph{perfect} if every vertex of $G$ is incident with an edge of $M$.

If $\chi'(G)$ equals the maximum degree of $G$, then $G$ is said to be \emph{class $1$}; otherwise $G$ is \emph{class $2$}. 
If $\chi'(G)=r$, then $r$ is
the minimum number such that $E(G)$ decomposes into $r$ matchings,
which are perfect matchings in case of $r$-regular graphs. For $r \geq 1$,
let $\ca T_r$ be the set of the smallest $r$-graphs of class 2. 
For example, the only element of $\ca T_3$ is the Petersen graph, which is
denoted by $P$ throughout this paper. 

The generalized Berge-Fulkerson Conjecture \cite{seymour1979multi} states that every $r$-graph
has $2r$ perfect matchings such that every edge is in precisely two of them. For $r=3$ the conjecture  was attributed to Berge and Fulkerson \cite{fulkerson1971blocking},
who put it into print (cf.~\cite{seymour1979multi}).
As a unifying approach to study some hard conjectures 
on cubic graphs, Jaeger \cite{jaeger1988nowhere} introduced colorings with
edges of another graph. To be precise, let $G$ and $H$ be graphs. An \emph{$H$-coloring} of $G$ is a mapping $f\colon E(G) \to E(H)$ such that
\begin{itemize}
	\item if $e_1,e_2 \in E(G)$ are adjacent, then $f(e_1) \neq f(e_2)$,
	\item for every $v \in V(G)$ there exists a vertex $u \in V(H)$ with $f(\partial_G(v))=\partial_H(u)$.
\end{itemize}

If such a mapping exists, then we write $H \prec G$ and say $H$ \emph{colors} $G$.
%Some substructures of $H$ induce substructures in $G$.For instance, if $H_1$ is a $k$-regular subgraph of $H$, then $G[f^{-1}(E(H_1))]$ is a $k$-regular subgraph of $G$,or if $H_2$ is a spanning subgraph of $H$, then $G[f^{-1}(E(H_2))]$ is aspanning subgraph of $G$. 
A set $\ca A$ of connected $r$-graphs such that for every connected $r$-graph $G$ 
there is an element $H \in \ca A$ which colors $G$ is said to be $r$-\emph{complete}.
For every $r\geq 3$, let $\ca H_r$ be an inclusion-wise minimal $r$-complete set.  

For $r=3$, Jaeger \cite{jaeger1988nowhere} conjectured that the Petersen graph 
colors every bridgeless cubic graph. If true, this conjecture would have far reaching 
consequences. For instance, it would imply that the Berge-Fulkerson Conjecture
and the 5-Cycle Double Cover Conjecture (see \cite{C.-Q._Zhang_book}) are 
also true.   
The Petersen Coloring Conjecture is a starting point for research in several directions. 
Different aspects of it are studied and partial results are proved, see for instance \cite{DeVos_etal_2007, Haglund_Steffen_2014, Jaeger_5_edge_coloring, Jin_partial_2021, Giuseppe_Vahan_2020, Riste_etal_2020, Robert_2017}. 

Analogously to the case $r=3$, if all elements of $\ca H_r$ would satisfy the generalized
Berge-Fulkerson Conjecture, then every $r$-graph would satisfy it.
Mazzuoccolo et al.~\cite{MTZ_r_graphs} asked 
whether there exists a connected $r$-graph $H$ such that $H \prec G$ for every (simple)
$r$-graph $G$, for all $r \geq 3$. 
We show that $\ca H_r$ is unique and that it is an infinite set when $r \geq 4$. Furthermore,
if $r=3$, then either $\ca H_3 = \{P\}$ (if the Petersen Coloring Conjecture is true) 
or $\ca H_3$ is an infinite set. More precisely, in Section \ref{Sec: H-coloring} we characterize $\ca H_r$ and provide constructions for infinite subsets of $\ca H_r$. Similar results are proved 
for simple $r$-graphs.

By definition, any $r$-graph $ G $ of class $1$ can be colored with any $r$-graph $ H $. Indeed, let $ M_1,\ldots, M_r$ be $ r $ pairwise disjoint perfect matchings of $ G $ and $ v $ a vertex of $ H $ with $ \partial_{H}(v)=\{e_1,\ldots,e_r\}$.
Every edge of $ M_i $ of  $ G $ can be mapped to $ e_i $ in $ H $.
Hence, the aforementioned questions and conjectures reduce to $r$-graphs of class $2$. In Section~\ref{Sec: smallest r-graphs} we determine the
set $\ca T_r$ of the smallest $r$-graphs of class 2
and prove that $|\ca T_r| \geq p'(r-3,6)$, where $p'(r-3,6)$ is the number 
of partitions of $r-3$ into at most $6$ parts. Furthermore, we show that 
if $r \geq 4$, then $\ca T_r$ is a proper subset of $\ca H_r$.   
 
The Petersen Coloring Conjecture has also been studied in the context of
quasi-orders on the set of graphs, see \cite{DeVos_etal_2007, Robert_2017}. In Section \ref{Sec: final remarks} we briefly put our 
results in this context. We conclude the paper with some open questions.

\subsection{Definitions and basic results}

Let $ G $ be a graph.
For any subset $ X $ of $ V(G) $, we  use $ G-X $ to denote the graph obtained from $ G $ by deleting all vertices of $ X $ and all incident edges. Similarly, for $ F \subseteq E(G) $, denote by $ G-F $ the graph obtained by deleting all edges of $F$ from $ G $.  In particular, we simply write 
$ G-x $ and $ G-e $ for $ G-X $ and $ G-F $, respectively, when $ X=\{x\} $ and $ F=\{e\} $.  The subgraph of $ G $ induced by the vertex set $ X $ is denoted by $ G[X] $. Moreover, the graph obtained from $ G $ by identifying all vertices of $ X $ and deleting all resulting loops is denoted  by $ G/X $;  we denote the new vertex by $w_X$.
Let $Y$ be a subset of $ V(G) $ with $ X\cap Y=\emptyset $. We use $ [X,Y]_G $ to denote the set of all edges of $ G $ with one vertex in $ X $ and the other one  in $Y$. %If  $ X $ or $ Y $ consists of one vertex, we skip the set-brackets notation.
Furthermore, if $ Y=X^c=V(G)\setminus X $ and $ [X,Y]_G $ is nonempty, 
then we call it an \emph{edge-cut} of $ G $ and denote it by $\partial_G(X)$. %(if $X$ consist of a single vertex, we skip the set-brackets notation). 
If  $ X $ or $ Y $ consists of one vertex, we skip the set-brackets notation.
In addition,  $|\partial_{G}(x)| $ is called the \emph{degree} of $ x \in V(G)$ and it is denoted by $d_G(x)$.
If $G$ is an $r$-graph, then $\partial_G(X)$
is \emph{tight} if $|X|$ is odd and $|\partial_G(X)|=r$. A tight edge-cut is \emph{trivial} if   $X$ or $X^c$ consists of a single vertex. Moreover, for $v \in V(G)$ we denote by $N_G(v)$ the set of neighbors of $v$.

A \emph{$ 1 $-factor}  of a graph $G $ is a spanning $ 1 $-regular subgraph of  $ G$, and its edge set is a perfect matching.
A connected $ 2 $-regular graph is called a  \emph{circuit}. A circuit of length $k$ is
called a \emph{$ k $-circuit} and it is denoted by $C_k$. 

For two graphs $ G $ and $ H $, if there are two bijections $ \theta: V(G) \to V(H)$ and $ \phi:E(G)\to E(H)  $ such that  $ e=uv\in E(G) $ if and only if $ \phi (e)=\theta(u)\theta(v)\in E(H)  $, then we say that $ G $ and $ H $ are {\em isomorphic}, denoted by $ G\cong H $, and call the pair of mappings $ (\theta,\phi) $ an  {\em isomorphism} between $ G $ and $ H $. In particular, an {\em automorphism} of a graph is an isomorphism of the graph to itself.

Let   $ H_1, \ldots, H_t $ be a sequence of graphs such that $V( H_i)\subseteq V(H_1) $ for each $ i\in\{2,\ldots,t\} $.  Denote by $H_1+E(H_2)+\ldots+E(H_t)$ the graph obtained from $H_1$ by adding a copy of every edge of $H_i$ for every $i\in\{2,\dots,t\}$. 
Let $\ca M$ be a finite multiset of perfect matchings of the Petersen graph $P$. 
The graph $P+\sum_{M\in\ca M}M$ is denoted by $P^{\ca M}$.

\begin{lem} [\cite{Grunewald_Steffen_1999}] \label{lem: P^M class 2}
For every finite multiset $\ca M$ of perfect matchings of the Petersen graph $P$,
the graph $P^{\ca M}$ is class 2.
\end{lem}

The following observation  will  frequently   be used without reference.

\begin{obs}\label{Observation-same-parity}
Let $r\geq 3$, let $G$ be an $r$-graph and let $X \subseteq V(G)$. If $\vert X \vert$ is even, then $\vert \partial_G(X) \vert$ is even. If $\vert X \vert$ is odd, then $\vert \partial_G(X) \vert$ has the same parity as $r$.
\end{obs}

One major fact that we use in this paper is that every $r$-graph can be
decomposed into a $k$-graph which is class 1 and an $(r-k)$-regular graph, for a suitable $k \in \{1, \dots ,r\}$. 
For every $r$-graph $G$ let $\pi(G)$ be the largest integer $t$ such that $G$ has $t$ pairwise disjoint perfect matchings.
Let $r\ge3$ and $k\in \{1,\dots, r\}$ be integers. Let 
$\ca G(r,k)=\{G\colon G$ is an $r$-graph with $\pi(G)=k \}$. 
Note that $\ca G(r,r-1) = \emptyset$, since
every $r$-graph with $r-1$ pairwise disjoint perfect matchings is a class 1 graph
and thus, it has $r$ pairwise disjoint perfect matchings. If $ k \leq r-2$, then the elements
of $\ca G(r,k)$ are class 2 graphs and $\ca G(r,i) \cap \ca G(r,j) = \emptyset$, if $1 \leq i \not = j \leq r-2$. We are interested in 
 the subset of $\ca G(r,k)$ consisting of all such graphs with the smallest order.
 This set is denoted by $\ca T(r,k)$.
 By definition, $\ca T_r \subseteq \bigcup_{i=1}^{r-2} \ca T(r,i)$.

\section{Smallest $r$-graphs of class $2$} \label{Sec: smallest r-graphs}

\subsection{Determination of $\ca T_r$}

The following theorem extends Lemma \ref{lem: P^M class 2} and characterizes the perfect matchings $M$ on $V(P)$ such that $P+M$ is a class $2$ graph.

\begin{theo}\label{theo:P+matching}
	Let $P$ be the Petersen graph and $H$ be a 1-regular graph on 
	$V(P)$ with edge set $M$. Then $P+M$ is class $2$ if and only if $M\subseteq E(P)$.
\end{theo}
\begin{proof} Lemma  \ref{lem: P^M class 2} has shown that $M\subseteq E(P)$ is a sufficient condition for $P+M$  to be class $ 2 $. We establish its necessity by way of contradiction. Suppose that there exists an edge  $ e \in M\setminus E(P) $. Let $ H_1=P+M$.
	%	It suffices to prove that if $P+M$ is class $2$, then $M\subseteq E(P)$.
	%	The other direction is true by Lemma \ref{lem: P^M class 2}. Let $ M $ be 
	%	the edge set of a 1-regular graph $H$ on $V(P)$ such that  $H_1=P+M$ is class $2$. Suppose that there exists an edge $ e \in M\setminus E(P) $. 
	Since any two  vertices of the Petersen graph are in a $ 5 $-circuit, the subgraph $ P $ of $ H_1 $ can be decomposed into two $ 5 $-circuits, $ C^1_5 $ and $C^2_5 $, and a 1-factor $H'  $ such that $e$ is a chord of $ C^1_5$ in $ H_1 $. Without loss of generality,  we assume $ C^1_5=u_1u_2u_3u_4u_5u_1$ with $e=u_2u_5$, as shown in Figure \ref{fig: C5+e}. 
	Let $H_2=H_1-E(H')=P+M-E(H')$. Note that $ H_2 $ is $ 3 $-regular and contains $C^1_5$ and $C^2_5$.
	If $| \partial_ {H_2}(V(C^1_5))|\neq1$, then $ H_2 $ is $ 2 $-edge-connected. This implies that $ H_2 $  is class 1 since it is not isomorphic to $P$, as it contains a $ 4 $-circuit $u_2u_3u_4u_5u_2 $. So, $ H_1= H_2+E(H')$ is also class 1, a contradiction.  Therefore, we may assume $| \partial_ {H_2}(V(C^1_5))|=1$ and set $ \partial_ {H_2}(V(C^1_5))=\{e'\} $.  The remaining proof is split into two cases. First, if $ e' $ is incident with $u_1$,
	then $M  $ contains an edge incident with $u_3$ and $ u_4 $.   
	Thus, $H_3=H_1-M_1$ contains a $ 3 $-circuit $u_1u_2u_5u_1$, a $ 2 $-circuit $u_3u_4u_3$ and a $ 5 $-circuit $ C^2_5 $, where $M_1=(M\setminus \{u_2u_5,u_3u_4\})\cup\{u_2u_3,u_4u_5\}$. Moreover, there are five edges between $ V(C^1_5) $ and $ V(C^2_5) $ in $ H_3 $, which implies that $ H_3 $ is $ 2 $-edge-connected. 
	Thus, $ H_3 $ is class 1 and so is $ H_1 $, a contradiction.  Second, 
	if $ e' $ is incident with $ u_3 $ or $ u_4 $, then, without loss of generality, we assume that $ e' $ is incident with $ u_3 $, and so  $M$ contains the edge $u_1u_4$.  Let $M_2=(M\setminus\{u_1u_4,u_2u_5\})\cup\{u_1u_2,u_4u_5\}$ and let $H_4=H_1-M_2$. 
	
	%The distance between two vertices of $P$ is either 1 or 2. Thus, $v_1$ is adjacent to $v_4$ in $P$, where
	There are two adjacent vertices $v_1$ and $v_4$ in $P$ such that $v_i \in N_{P}(u_i)\setminus V(C^1_5)$ for each $i\in\{1,4\}$. Then $H_4$ contains a $ 4 $-circuit $u_1u_4v_4v_1u_1$. Moreover, $ H_4 $ is $ 2 $-edge-connected since there are five edges between $ V(C^1_5) $ and $ V(C^2_5) $.  This implies that $H_4$  is class 1 and therefore, $H_1$ is also class 1, a contradiction. 
\end{proof}

\begin{figure}[htbp]
	\centering
	\scalebox{1}{%% Creator: Inkscape 1.2 (dc2aedaf03, 2022-05-15), www.inkscape.org
%% Accompanies image file 'C5_ev2.pdf' (pdf, eps, ps)
\begingroup%
  \makeatletter%
  \providecommand\color[2][]{%
    \errmessage{(Inkscape) Color is used for the text in Inkscape, but the package 'color.sty' is not loaded}%
    \renewcommand\color[2][]{}%
  }%
  \providecommand\transparent[1]{%
    \errmessage{(Inkscape) Transparency is used (non-zero) for the text in Inkscape, but the package 'transparent.sty' is not loaded}%
    \renewcommand\transparent[1]{}%
  }%
  \providecommand\rotatebox[2]{#2}%
  \newcommand*\fsize{\dimexpr\f@size pt\relax}%
  \newcommand*\lineheight[1]{\fontsize{\fsize}{#1\fsize}\selectfont}%
  \ifx\svgwidth\undefined%
    \setlength{\unitlength}{139.29602089bp}%
    \ifx\svgscale\undefined%
      \relax%
    \else%
      \setlength{\unitlength}{\unitlength * \real{\svgscale}}%
    \fi%
  \else%
    \setlength{\unitlength}{\svgwidth}%
  \fi%
  \global\let\svgwidth\undefined%
  \global\let\svgscale\undefined%
  \makeatother%
  \begin{picture}(1,0.78588845)%
    \lineheight{1}%
    \setlength\tabcolsep{0pt}%
    \put(0,0){\includegraphics[width=\unitlength,page=1]{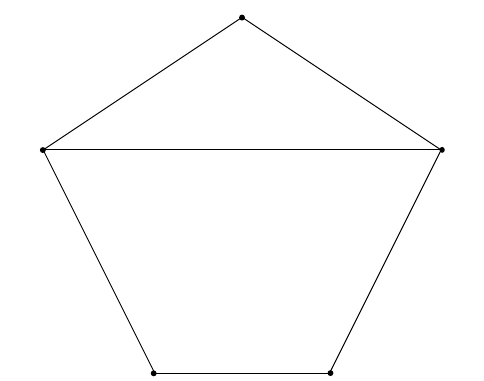}}%
    \put(0.48985459,0.49265317){\makebox(0,0)[lt]{\lineheight{1.25}\smash{\begin{tabular}[t]{l}$e$\end{tabular}}}}%
    \put(-0.00222725,0.46814327){\makebox(0,0)[lt]{\lineheight{1.25}\smash{\begin{tabular}[t]{l}$u_2$\end{tabular}}}}%
    \put(0.47706804,0.76550256){\makebox(0,0)[lt]{\lineheight{1.25}\smash{\begin{tabular}[t]{l}$u_1$\end{tabular}}}}%
    \put(0.92059197,0.47064011){\makebox(0,0)[lt]{\lineheight{1.25}\smash{\begin{tabular}[t]{l}$u_5$\end{tabular}}}}%
    \put(0.21913928,0.00632801){\makebox(0,0)[lt]{\lineheight{1.25}\smash{\begin{tabular}[t]{l}$u_3$\end{tabular}}}}%
    \put(0.6998391,0.00773603){\makebox(0,0)[lt]{\lineheight{1.25}\smash{\begin{tabular}[t]{l}$u_4$\end{tabular}}}}%
  \end{picture}%
\endgroup%
} 
	\caption{The $5$-circuit $C^1_5$ with the edge $ e $.}
	\label{fig: C5+e}
\end{figure}
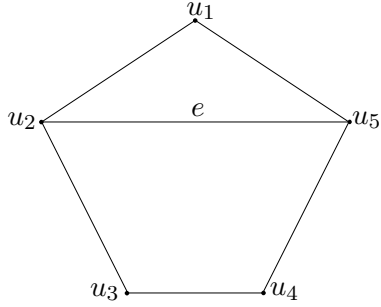

%xxxxxxxxxxxxxxxxxxxxxxxxxxxxxxxxxxxxxxxxxxxxxxxxxxxxxxxxxxxxxxxxx
%By definition, $\ca T_r \subseteq \bigcup_{i=1}^{r-2} \ca T(r,i)$.
%In this section we determine the elements of $\ca T_r$ and prove $\ca T_r = \ca T(r,r-2)$.

\begin{theo}\label{Theorem-cong-P+M_V2}
	For all $r\ge 3$,
	%	If $G$ is a smallest $r$-graph of class $2$,
	%	then $G \cong P^{\ca M}$ for a suitable multiset $\ca M$ of $r-3$ perfect matchings of $P$.
	$\ca T_r = \ca T(r,r-2) = \{P^{\ca M}\colon \ca M $ is a multiset of $r-3$ perfect matchings of the Petersen graph $ P\}$.
\end{theo}

\begin{proof}

We will deduce the statement from the following three claims. 

\begin{claim}
	\label{obs:SmallestClass2r-graphNoTightCut}
	Let $r\ge3$. If $G$ is a smallest $r$-graph of class $2$, then $G$ has no non-trivial tight edge-cut.
\end{claim}

\emph{Proof of Claim \ref{obs:SmallestClass2r-graphNoTightCut}.}
	Suppose that there is an odd set $X \subseteq V(G)$ such that $|\partial_G(X)|=r$ and neither $X$ nor $X^c$ consists of  a single vertex. By the minimality of $|V(G)|$, the $r$-graphs $G/X$ and $G/X^c$ are class 1. As a consequence, $G$ is also class 1, a contradiction.
\ENDproof

\begin{claim} \label{Lemma-order10-r-2-PDPM}
	Let $r\ge3$. If $G$ is a smallest $r$-graph of class $2$, then $|V(G)|=10$ and $G$ has $r-2$ pairwise disjoint perfect matchings.
\end{claim}

\emph{Proof of Claim \ref{Lemma-order10-r-2-PDPM}.}
We prove the claim by induction on $r$. When $r=3$, the statement follows from the fact that the smallest $ 3 $-graph of class $2$  is the Petersen graph. Hence, let $r\ge 4$ and assume the statement is true for every $r'<r$.

Let $G$ be a smallest $r$-graph of class $2$. By Lemma \ref{lem: P^M class 2}, $|V(G)|\leq 10$. Note that every $ r $-graph has a perfect matching \cite{seymour1979multi}. Thus, let $M$ be a perfect matching of $G$.

If $H=G-M$ is an $(r-1)$-graph, then $H$ is also class $2$, since otherwise $G$ would be class $1$. 
Furthermore, we have $  |V(G)|=|V(H)|\geq10 $ in this case, which implies $  |V(G)|=|V(H)|=10$. Thus,  the statement follows by induction .

Therefore, we may assume that $H=G-M$ is not an $(r-1)$-graph. 
By the definition and Observation \ref{Observation-same-parity},
there is an odd set $X\subseteq V(G)$ such that $|\partial_G(X)\setminus M|\le r-3$. Moreover, we have $|\partial_G(X)|\ge r+2$ by Claim \ref{obs:SmallestClass2r-graphNoTightCut}. Hence, $|\partial_G(X)\cap M| = |\partial_G(X)| - |\partial_G(X)\setminus M| \ge 5.$ Since $M$ is a perfect matching, we conclude that $|V(G)|=10.$  As a consequence, $M$ has cardinality $5$ and thus, $|\partial_G(X)\cap M|=5$ and $|\partial_G(X)|=r+2.$ Let $x_1y_1$ and $x_2y_2$ be two different edges of $\partial_G(X)\cap M$, where $x_1,x_2\in X$. The graph $G'=G-\{x_1y_1,x_2y_2\} + \{x_1x_2,y_1y_2\}$ is still an $r$-graph. Indeed, for any odd set $Y\subseteq  V(G')$ we have $|\partial_{G'}(Y)| \ge |\partial_G(Y)|-2 \ge r$. Moreover, $|\partial_{G'}(X)| =r$ and hence, $G'$ is class $1$ by Claim \ref{obs:SmallestClass2r-graphNoTightCut}. Let $\ca N$ be a set of $r$ pairwise disjoint perfect matchings of $G'$ and let $N_x$ and $N_y$ be the perfect matchings containing $x_1x_2$ and $y_1y_2$ respectively (note that $N_x\ne N_y$ since otherwise $G$ itself would be class $1$). Then $\ca N \setminus \{N_x,N_y\}$ is a set of $r-2$ pairwise disjoint perfect matchings of $G$.
\ENDproof

\begin{claim}
\label{lem:P_subgraph}
Let $r\ge3$. If $G$ is a smallest $r$-graph of class $2$, then there is a set $\mathcal{M}$ of $r-3$ pairwise disjoint perfect matchings of $G$ such that $G - \bigcup_{M\in \ca M} M \cong P$.
\end{claim}

\emph{Proof of Claim \ref{lem:P_subgraph}.}
We prove the claim by induction on $r$. When $r=3$, the statement is trivial since the smallest $ 3 $-graph of class 2 is the Petersen graph. Hence, let $r\ge 4$ and assume the statement is true for every $r'<r$.

Let $G$ be a smallest $r$-graph of class 2. By Claim \ref{Lemma-order10-r-2-PDPM}, $G$ is of order 10 and has a set $\ca N$ of $r-2$ pairwise disjoint perfect matchings. Let $M \in \ca N$. Then 
$ G-M $ is class $ 2 $, since otherwise $ G $ would be class $ 1 $. If $G-M$ is an $(r-1)$-graph, then the statement follows by induction. Hence, there exists an odd set $X\subseteq V(G-M)$ with $ |\partial_{G-M}(X)| \leq  r-3$. 
Furthermore, $V(G-M)=V(G) $ and $|\partial_G(X)\setminus M| =|\partial_{G-M}(X)| $. By Claim \ref{obs:SmallestClass2r-graphNoTightCut} and Claim \ref{Lemma-order10-r-2-PDPM}, we have $|\partial_G(X)|\geq r+2$ and $ |M|=5 $. As a consequence,  we obtain $|\partial_G(X)|= r+2$ and $|\partial_G(X)\cap M| =  5$, which implies $|X|=5$. Set $H=G - \cup_{N\in \ca N} N$ and note that $H$ is a $2$-factor of $G$, which contains at least two odd circuits, since otherwise $G$ would be class 1. Every perfect matching of $\ca N$ contains at least one edge of $\partial_G(X)$ and hence, $|\partial_H(X)| =  0$. Thus, both $H[X]$ and $H[X^c]$ either consists of a $5$-circuit or a $3$-circuit and a $2$-circuit. We consider the following two cases.
	
{\bf Case 1.}\ $H+M$ is a $3$-graph.
	
In this case $H+M \cong P$, since otherwise $H+M$ is class $1$ which would imply that $G$ is also class 1.

{\bf Case 2.}\ $H+M$ is not a $3$-graph.
	
Thus, $H+M$ has a bridge, which implies that both $H[X]$ and $H[X^c]$ consists of  a $3$-circuit and a $2$-circuit and $|\partial_{H+M}(V(C) \cup V(C'))| =  1$, where $C$ is the $3$-circuit of $H[X]$ and $C'$ is the $2$-circuit of $H[X^c]$. As a consequence, there is only one possibility for the structure of $G+M$, which is depicted in Figure \ref{fig:H+M}. With respect to the vertex labels in Figure \ref{fig:H+M}, set $M'=\left(M\setminus \{v_1v_4, v_2v_3\}\right)\cup \{v_1v_2, v_3v_4\}$ and $\ca N' = \left( \ca N \setminus \{M\} \right) \cup \{M'\}$. Then, $\ca N'$ is a set of $r-2$ pairwise disjoint perfect matchings of $G$. Now, consider $\ca N'$ and $M'$ instead of $\ca N$ and $M$, respectively, and repeat the same arguments as above. We deduce that $G-M'$ is an $(r-1)$-graph and the statement follows by induction.
\ENDproof

\begin{figure}[htbp]
\centering
\scalebox{1}{%% Creator: Inkscape 1.2 (dc2aedaf03, 2022-05-15), www.inkscape.org
%% Accompanies image file 'H+Mv2.pdf' (pdf, eps, ps)
\begingroup%
  \makeatletter%
  \providecommand\color[2][]{%
    \errmessage{(Inkscape) Color is used for the text in Inkscape, but the package 'color.sty' is not loaded}%
    \renewcommand\color[2][]{}%
  }%
  \providecommand\transparent[1]{%
    \errmessage{(Inkscape) Transparency is used (non-zero) for the text in Inkscape, but the package 'transparent.sty' is not loaded}%
    \renewcommand\transparent[1]{}%
  }%
  \providecommand\rotatebox[2]{#2}%
  \newcommand*\fsize{\dimexpr\f@size pt\relax}%
  \newcommand*\lineheight[1]{\fontsize{\fsize}{#1\fsize}\selectfont}%
  \ifx\svgwidth\undefined%
    \setlength{\unitlength}{234.28068169bp}%
    \ifx\svgscale\undefined%
      \relax%
    \else%
      \setlength{\unitlength}{\unitlength * \real{\svgscale}}%
    \fi%
  \else%
    \setlength{\unitlength}{\svgwidth}%
  \fi%
  \global\let\svgwidth\undefined%
  \global\let\svgscale\undefined%
  \makeatother%
  \begin{picture}(1,0.32520672)%
    \lineheight{1}%
    \setlength\tabcolsep{0pt}%
    \put(0,0){\includegraphics[width=\unitlength,page=1]{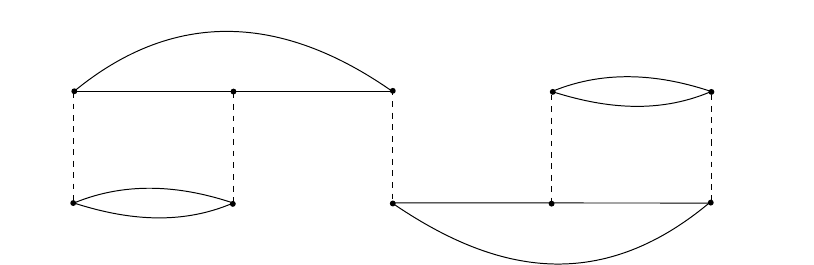}}%
    \put(0.04282678,0.06789375){\makebox(0,0)[lt]{\lineheight{1.25}\smash{\begin{tabular}[t]{l}$v_4$\end{tabular}}}}%
    \put(0.29252727,0.06789375){\makebox(0,0)[lt]{\lineheight{1.25}\smash{\begin{tabular}[t]{l}$v_3$\end{tabular}}}}%
    \put(0.26691694,0.22475688){\makebox(0,0)[lt]{\lineheight{1.25}\smash{\begin{tabular}[t]{l}$v_2$\end{tabular}}}}%
    \put(0.04602807,0.20875043){\makebox(0,0)[lt]{\lineheight{1.25}\smash{\begin{tabular}[t]{l}$v_1$\end{tabular}}}}%
    \put(0.95199311,0.22795817){\color[rgb]{0,0,0}\makebox(0,0)[lt]{\lineheight{0.6250006}\smash{\begin{tabular}[t]{l}$X$\end{tabular}}}}%
    \put(0,0){\includegraphics[width=\unitlength,page=2]{H+Mv2.pdf}}%
  \end{picture}%
\endgroup%
} 
\caption{The graph $H+M$ in Case 2 of the proof of Claim \ref{lem:P_subgraph} (Theorem \ref{Theorem-cong-P+M_V2}). The dashed edges belong to $M$.}
\label{fig:H+M}
\end{figure}

By Claim \ref{Lemma-order10-r-2-PDPM},
	we have $\ca T_r = \ca T(r,r-2)$. Moreover, by Theorem \ref{theo:P+matching}
	and Claim \ref{Lemma-order10-r-2-PDPM},
	for any multiset $\ca M $ of $r-3$ perfect matchings of $ P$,  the graph $ P^{\ca M} $ is in $ \ca T_r $. It remains to show that, if
 $ G \in\ca T_r$, then $ G\cong  P^{\ca M} $ for a suitable multiset $\ca M $.
By Claim \ref{lem:P_subgraph}, there is a set $\ca N$ of $r-3$ pairwise disjoint perfect matchings of $G$ such that the graph $H=G - \bigcup_{N\in \ca N} N$ is isomorphic to the Petersen graph. For every $N \in \ca N$, the graph $H+N$ is class 2, since otherwise $G$ is class 1. Therefore,  $ G\cong P^{\cal N} $
 by Theorem \ref{theo:P+matching}.
\end{proof}

\subsection{Lower bounds for $|\ca T_r|$}

The following lemma is a direct consequence of the fact that the Petersen graph is 
3-arc-transitive, see e.g.~Corollary 1.8 in \cite{Babai_Handbook}.
That is, for any two paths of length 3 of $P$ there is an
automorphism of $P$ which maps one to the other.

\begin{lem}\label{lem:autP_fixing_3_pm}
	Let $M_1,\dots,M_6$ be the six perfect matchings of the Petersen graph $P$. Moreover, let $N_1,N_2,N_3 \in \{M_1,\dots,M_6\}$ and $g\colon \{N_1,N_2,N_3\} \to \{M_1,\dots,M_6\}$ be an injective function. There is an automorphism $ (\theta,\phi) $ of $P$ such that, for all $i\in\{1,2,3\}$, $\phi(N_i)=g(N_i)$.
\end{lem}

\begin{proof}
	Let $N_1,N_2$ and $N_3$ be pairwise different perfect matchings of $P$. If we prove the statement in this case then the proof is complete.
	
	Note that the unique edge $x_1x_2$ in $N_1\cap N_2$ and the unique edge $x_3x_4$ in $N_1 \cap N_3$ are at distance one, i.e.\ the subgraph $P[\{x_1,x_2,x_3,x_4\}]$ is a path $T$ on four vertices. Up to changing names to such vertices, we may assume that $T=x_1x_2x_3x_4$. The same holds for the unique edge $y_1y_2$ in $g(N_1)\cap g(N_2)$ and the unique edge $y_3y_4$ in $g(N_1) \cap g(N_3)$. Without loss of generality, we can assume again that $y_1y_2y_3y_4$ is a path on four vertices.
	
	Since $P$ is 3-arc-transitive there is an automorphism $ (\theta,\phi) $ of $P$ such that, for all $i\in\{1,\dots,4\} $, $ \theta(x_i)=y_i$. Since $ (\theta,\phi) $ is an automorphism, $\phi(N_1)$ must be a perfect matching. Moreover, since the only perfect matching of $P$ containing both $y_1y_2$ and $y_3y_4$ is $g(N_1)$ we get $\phi(N_1)=g(N_1)$.
	
	Similarly, $\phi(N_2)$ and $\phi(N_3)$ are perfect matchings of $P$ different from $\phi(N_1)$, such that $y_1y_2\in \phi(N_2)$ and $y_3y_4 \in \phi(N_3)$. Then, the only possibility is that $\phi(N_2)=g(N_2)$ and $\phi(N_3)=g(N_3).$
\end{proof}

We now consider partitions of integers, which are ways of writing an integer as a sum of positive integers, see e.g.~\cite{matouvsek2008invitation}. 
We are interested in partitions of an integer into a fixed number of parts.
We allow $0$ to be a part of a partition.
A \emph{partition} of an integer $n$ into $k$ parts is a multiset of $k$ integers
$n_1, \dots,n_k$ with $n_i \geq 0$ for $i \in \{1, \dots,k\}$ 
such that $n = \sum_{i=1}^k n_i$. Two partitions of $n$ are equal if
they yield the same multiset, i.e.\ if they differ only in the order of their elements. 
For two positive integers $k\le n$, let $p'(n,k)$ be the number of partitions of $n$ into $k$ parts. Set $p'(0,k)=1$.   

\begin{theo} \label{thm: lower bound S(r,r-2)}
If $3\leq r \leq 8$, then $|\ca T_r| = p'(r-3,6)$, and 
	if $r \geq 9$, then  $|\ca T_r| > p'(r-3,6)$.
\end{theo}

\begin{proof} 
By Theorem \ref{Theorem-cong-P+M_V2}, 
any graph $G \in \ca T_r$ can be expressed as $G = P + \sum_{i=1}^6 n_i M_i$, where $M_1, \dots, M_6$ are the six pairwise different perfect matchings of $P$. In this case, $n_1, \dots, n_6$ is a partition of $r-3$ into six parts. We say that $G$ \emph{induces} this partition of $r-3$. 

\begin{claim}\label{claim: 1 min class 2}
Let $r \geq 3$ be an integer and
$G, G' \in \ca T_r$. If $G \cong G'$, then $G$ and $G'$ induce the same partition of $r-3$.
\end{claim}

\emph{Proof of Claim \ref{claim: 1 min class 2}.}
We can assume that $G = P + \sum_{j=1}^6 n_j M_j$ and $G' = P + \sum_{j=1}^6 n_j' M_j$.
%, where $M_1, \dots, M_6$ are the six pairwise different perfect matchings of $P$.
 For the subgraph $ P $ of $ G $ and $ G' $, we label an edge $e$ of $P$ by the set $\{p,q\}$ if $M_p \cap M_q = \{e\}$, $p \not = q$. 
Then all possible labels are used and no two edges receive the same label in $ P $.

Since $G \cong G'$, there is an isomorphism  between $G$ and $G'$ which maps the labeled edge
$\{p,q\}$ of $ G $ to a labeled edge $\{i_p,i_q\}$ of $ G' $ for each $ \{p,q\} \subseteq\{1,\ldots,6\} $. Furthermore, $n_p + n_q = n_{i_p}'+n_{i_q}'$.
Thus, for $\{1,2\}, \dots, \{1,6\}$, we get that $4n_1 + \sum_{j=1}^6 n_j = 4n_{i_1}' + \sum_{j=1}^6 n_{i_j}'$. Since 
$\sum_{j=1}^6 n_j = \sum_{j=1}^6 n_{i_j}' = r-3$,  it follows that $n_1 = n_{i_1}'$. With similar arguments, we further obtain that $n_j = n_{i_j}'$ for each $j \in \{1, \dots, 6\}$.
\ENDproof

%\yl{In the original version, it seems that all $ M_1 $ are mapped into the same kind of perfecting, i.e. $ M_{i_1} $. A bit confusing, so I add some details.}

\begin{claim} \label{claim: 2 min class 2}
If $r \geq 9$, then there are non-isomorphic graphs in $\ca T_r$ which induce the same partition. 
\end{claim}

\emph{Proof of Claim \ref{claim: 2 min class 2}.} Let $N_1,\dots,N_4$ be four pairwise different perfect matchings of $P$ such that the edge in $N_1\cap N_2 = \{uv\}$ is adjacent to the edge in $N_3\cap N_4=\{uz\}$. There is a fifth perfect matching $N_5$ of $P$ such that the unique edge in $N_3\cap N_5$ is not adjacent to $uv$. 

Let $t\ge2$ be an integer and consider the $(t+7)$-graphs $G_t^1 = P + tN_1 + 2N_2 + N_3 + N_4$ and $G_t^2 = P + tN_1 + 2N_2 + N_3 + N_5$. Note that both $G_t^1$ and $G_t^2$ have exactly one pair of vertices connected by $ t+3 $ edges, i.e.\ $ |[u,v]_{G^1_t} |=|[u,v]_{G^2_t}|=t+3$.
 On one hand, $uv$ is adjacent to $uz$ and $|[u,z]_{G^1_t}|=3$. On the other hand, by the choice of $N_5$, $uv$ is adjacent only to edges $xy$ such that $|[x,y]_{G_t^2}|\le 2$. We deduce that $G_t^1 \not\cong G_t^2.$
\ENDproof

%
%Let $G$ be a graph and $f\colon V(G)\to V(G)$ be an automorphism. For $F\subseteq E(G)$, we denote by $f(F)$ the set $\{f(x)f(y)\colon xy \in F \}$. Note that, if $f$ is an automorphism of $P$ and $M\subseteq E(P)$ is a perfect matching, then $f(M)$ is also a perfect matching.

\begin{claim}\label{claim: 3 min class 2}
	Let $r \leq 8$ and
	$G, G' \in \ca T_r$. If $G$ and $G'$ induce the same partition of $r-3$,
	then $G \cong G'$.
\end{claim}

\emph{Proof of Claim \ref{claim: 3 min class 2}.} 
%Let $M_1,\dots,M_6$ be the six perfect matchings of the Petersen graph $P$. 
Assume that $ G = P^{\ca M}= P + \sum_{j=1}^6 n_j M_j$ and 
$G'=P^{\ca M'} = P + \sum_{j=1}^6 n_j' M_j$ induce the same partition of $r-3$. Let $\ca M_0 = \{M_j \colon n_j\ne0\}$ and $\ca M_0' = \{M_j \colon n_j'\ne0\}$. Then $|\ca M_0 |= |\ca M'_0| $.

If $|\ca M_0|\le 3$, choose a bijection $g\colon \ca M_0 \to \ca M'_0$ such that $g(M_{\alpha})=M_{\beta}$ if and only if $n_{\alpha} = n_{\beta}'$. By Lemma \ref{lem:autP_fixing_3_pm}, there is an automorphism $(\theta,\phi)$ of $P$ such that, for each perfect matching $N\in\ca M_0$, $\phi(N) = g(N)$. It follows that $(\theta,\phi')$ is an isomorphism of $P^{\ca M}$ to $P^{\ca M'}$, where $ \phi'(M_i)=\phi(M_i) $  for each $ i\in\{1,\ldots,6\} $.
The only other cases are the following.
\begin{itemize}
	\item $r-3 =4$ with partition $1,1,1,1,0,0$;
	\item $r-3 =5$ with partitions $2,1,1,1,0,0$ or $1,1,1,1,1,0.$
\end{itemize}
In such cases, we let $\ca M_1 = \{M_j \colon n_j=1\}$ and $\ca M_1' = \{M_j \colon n_j'=1\}$. Let $\ca N_1$ be the set of perfect matchings of $P$ different from those of $ \ca M_1$ and $\ca N'_1$ be the set of perfect matchings of $P$ different from those of $\ca M'_1$. Then, there is a bijection $g\colon \ca N_1 \to \ca N'_1$ such that $g(M_{\alpha})=M_{\beta}$ if and only if $n_{\alpha} = n_{\beta}'$. The proof now, follows as above. Namely, since $|\ca N_1|=|\ca N_1'|\le3 $, by Lemma \ref{lem:autP_fixing_3_pm}, there is an automorphism $(\theta,\phi)$ of $P$ such that, for all $N\in\ca N_1$, $\phi(N) = g(N)$. Then, $(\theta,\phi')$ is an isomorphism of $P^{\ca M}$ to $P^{\ca M'}$, where $ \phi'(M_i)=\phi(M_i) $ for each $ i\in\{1,\ldots,6\} $.
\ENDproof

By Claims \ref{claim: 1 min class 2}, \ref{claim: 2 min class 2} and \ref{claim: 3 min class 2}, the theorem is proved.
\end{proof}

\section{Complete sets} \label{Sec: H-coloring}

In this section we give the following characterization of $\ca H_r$: $G\in \ca H_r$ if and only if the only connected $r$-graph coloring $G$ is $G$ itself. Moreover, we show that $\ca H_r$ is an infinite set when $r\ge 4$. For $r=3$ it turns out that, if the Petersen Coloring Conjecture is false, then $\ca H_3$ is an infinite set too.
%In this section we prove lower bounds for $|\ca H_r|$ and thus, for $|\ca S_r|$.
We prove similar results for the restriction on simple $r$-graphs.

We start with some preliminary technical results. In particular, we introduce a lifting operation for $r$-graphs.

\subsection{Substructures and lifting}

Let $G$ be a graph and $F \subseteq E(G)$. We say that $F$ \emph{induces} a subgraph $H$
of $G$ if $E(H) = F$ and $V(H)$ contains all vertices of $G$ which are incident with
an edge of $F$. We denote such a subgraph $H$ by $G[F]$. 
A spanning subgraph $ G' $ of $ G $ is a 
$ \{K_{1,1}, C_m\colon m\geq3\} $-factor if each component of $ G' $ is isomorphic to an element of $ \{K_{1,1}, C_m\colon m\geq3\} $, where $ K_{s,t} $ is the complete bipartite graph with two partition sets of sizes $ s $ and $ t $.
Some of the following observations appear also in  \cite{MTZ_r_graphs}.

\begin{obs}
	\label{obs:coloring_basics}
	Let $H$ and $G$ be graphs and let $f$ be an $H$-coloring of $G$.
	\begin{itemize}
		\item[(i)] $\chi'(G) \leq \chi'(H)$.
		\item[(ii)] If $M_1,\dots, M_k$ are $k$ pairwise disjoint perfect matchings in $H$, then $f^{-1}(M_1),\dots, f^{-1}(M_k)$ are $k$ pairwise disjoint perfect matchings in $G$.
		\item[(iii)] If $C$ is a $2$-regular subgraph of $H$, then $f^{-1}(E(C))$ induces a $2$-regular subgraph in $G$.
		\item[(iv)] If $ H' $ is a $ \{K_{1,1}, C_m\colon m\geq3\} $-factor in $H$, then $ f^{-1} (E(H'))$ induces a $ \{K_{1,1}, C_m\colon m\geq3\} $-factor in $ G $.
	\end{itemize}
\end{obs}
\begin{proof}
Let $ H' $ be a subgraph of $ H $ and $ G' $ be the subgraph of $ G $ induced by $f^{-1}(E(H'))$. 
By the definition of $ H $-coloring,  if $ H' $ is $ k $-regular (spanning, respectively) then $ G' $ is $ k $-regular (spanning, respectively). Then statements $ (i),(ii) $ and $ (iii) $ can be obtained immediately. 
In order to show statement $ (iv) $, assume that $ H' $ is a $ \{K_{1,1}, C_m\colon m\geq3\} $-factor.
We decompose $ H' $ into a $ 1 $-regular subgraph $ H_1 $ and a $ 2 $-regular subgraph $ H_2 $. The sets $ f^{-1} (E(H_1))$ and $ f^{-1} (E(H_2))$ induce a $ 1 $-regular subgraph $ G_1 $ and a $ 2 $-regular subgraph $ G_2 $ of $ G $, respectively. By the definition of $ H $-coloring, $ G_1 $ and $ G_2 $ are disjoint.
This completes the proof.
\end{proof}

Let $G$ be a graph and let $x \in V(G)$ with $ |N_G(x)|\geq2 $. A \emph{lifting} (of $G$) at $x$ is the following operation: Choose two distinct neighbors $y$ and $z$ of $x$, delete an edge $e_1$ connecting $x$ with $y$, delete an edge $e_2$ connecting $x$ with $z$ and add a new edge $e$ connecting $y$ with $z$; additionally, if $ e_1 $ and $ e_2 $ were the only two edges incident with $ x $, then delete the vertex $ x $ in the new graph. We say $e_1$ and $e_2$ are \emph{lifted to} $e$; the new graph is denoted by $G(e_1,e_2)$. 

We will make use of the following fact. Let $G$ be a graph, then 
$\vert \partial_G(X \cap Y) \vert + \vert \partial_G(X \cup Y) \vert \leq \vert \partial_G(X) \vert + \vert \partial_G(Y) \vert$ for every $X,Y \subseteq V(G)$.

\begin{lem}\label{Lem:r-graph lifting}
	Let $ r \geq 2$ be an integer and let $G$ be a connected graph of order at least $ 2 $ with a vertex $x \in V(G)$ such that
	\begin{itemize}
		\item $d_G(v)=r$ for all $v \in V(G)\setminus\{x\}$, and
		\item if $\vert V(G) \vert$ is even, then $d_G(x)\neq r$, and
		\item $\vert \partial_G(S) \vert \geq r$ for every $S\subseteq V(G)\setminus\{x\}$ of odd cardinality.
	\end{itemize}
	Then, for every labeling $\partial_G(x)=\{e_1,\ldots, e_{d_G(x)}\}$ there exists an  $i \in \mathbb{Z}_{d_G(x)}$ such that $G(e_i,e_{i+1})$ is a connected graph with $\vert \partial_{G(e_i,e_{i+1})}(S') \vert \geq r$ for every $S'\subseteq V(G(e_i,e_{i+1}))\setminus\{x\}$ of odd cardinality.
\end{lem}

\begin{proof} We argue by contradiction. Let $G$ be a possible counterexample of smallest order, let $d=d_{G}(x)$, and let $e_i=xy_i$ for every $i \in \{1,\ldots, d\}$.
	
	First we show $\vert N_{G}(x) \vert \geq 2$. Suppose that $x$ has just one neighbor $x'$. Note that $d_{G}(x')=r$ by our assumptions. If $\vert V(G) \vert$ is even, then $d_{G}(x) \neq r$. As a consequence, the set $S=V(G) \setminus \{x\}$ is a set of odd cardinality with $\vert \partial_{G}(S) \vert = d_{G}(x) <r$, a contradiction. If $\vert V(G) \vert$ is odd, then the set $S=V(G) \setminus \{x, x'\}$ is a set of odd cardinality with $\vert \partial_{G}(S) \vert = r - d_{G}(x)<r$, a contradiction again. Therefore, $\vert N_{G}(x) \vert \geq 2$.
	
	Hence, we can choose an $i \in \mathbb{Z}_d$ such that $y_i \neq y_{i+1}$ and, if $G-x$ is not connected, then $y_i$ and $y_{i+1}$ belong to different components of $G-x$. Suppose that $G$ has a bridge $e$. Then, for parity reasons, the component $H$ of $G-e$ not containing $x$ is of odd order, a contradiction since $|\partial_{G}(V(H))| =1< r$. Thus, $G$ is bridgeless and hence, the graph $G(e_i,e_{i+1})$ is connected by the choice of $i$. As a consequence, there is a set $T \subseteq V(G(e_i,e_{i+1}))\setminus \{x\}$ of odd cardinality with $\vert \partial_{G(e_i,e_{i+1})}(T)\vert <r$, since $G$ is a counterexample. Observe that $ \vert \partial_{G}(T) \vert$ has the same parity as $ r $, which implies $\vert \partial_{G}(T) \vert =r$ and $y_i,y_{i+1} \in T$. Set $G_1=G/T$ and label the edges of $\partial_{G_1}(x)$ with the same labels as in $G$. Then, $G_1$ and $x$ satisfy the conditions of the statement. Therefore, by the minimality of $\vert V(G) \vert$, there is an integer $j \in \mathbb{Z}_d$ such that the graph $G_2=G_1(e_j,e_{j+1})$ satisfies $\vert \partial_{G_2}(S) \vert \geq r$ for every $S\subseteq V(G_2)\setminus\{x\}$ of odd cardinality. Set $G_3=G(e_j,e_{j+1})$. The graphs $G, G_1, G_2$ and $G_3$ are depicted in Figure~\ref{fig:lifting_Lemma}.

\begin{figure}[htbp]
\centering
\subfigure[$G$]{
\begin{minipage}[t]{0.225\textwidth}
\centering
%% Creator: Inkscape 1.2 (dc2aedaf03, 2022-05-15), www.inkscape.org
%% Accompanies image file 'G0_v2.pdf' (pdf, eps, ps)
\begingroup%
  \makeatletter%
  \providecommand\color[2][]{%
    \errmessage{(Inkscape) Color is used for the text in Inkscape, but the package 'color.sty' is not loaded}%
    \renewcommand\color[2][]{}%
  }%
  \providecommand\transparent[1]{%
    \errmessage{(Inkscape) Transparency is used (non-zero) for the text in Inkscape, but the package 'transparent.sty' is not loaded}%
    \renewcommand\transparent[1]{}%
  }%
  \providecommand\rotatebox[2]{#2}%
  \newcommand*\fsize{\dimexpr\f@size pt\relax}%
  \newcommand*\lineheight[1]{\fontsize{\fsize}{#1\fsize}\selectfont}%
  \ifx\svgwidth\undefined%
    \setlength{\unitlength}{91.08550978bp}%
    \ifx\svgscale\undefined%
      \relax%
    \else%
      \setlength{\unitlength}{\unitlength * \real{\svgscale}}%
    \fi%
  \else%
    \setlength{\unitlength}{\svgwidth}%
  \fi%
  \global\let\svgwidth\undefined%
  \global\let\svgscale\undefined%
  \makeatother%
  \begin{picture}(1,1.05002541)%
    \lineheight{1}%
    \setlength\tabcolsep{0pt}%
    \put(0,0){\includegraphics[width=\unitlength,page=1]{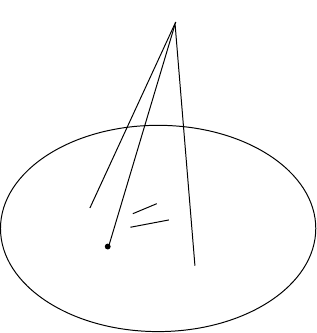}}%
    \put(0.51855912,1.01842225){\makebox(0,0)[lt]{\lineheight{1.25}\smash{\begin{tabular}[t]{l}$x$\end{tabular}}}}%
    \put(0,0){\includegraphics[width=\unitlength,page=2]{G0_v2.pdf}}%
    \put(0.15620724,0.38366753){\makebox(0,0)[lt]{\lineheight{1.25}\smash{\begin{tabular}[t]{l}$y_i$\end{tabular}}}}%
    \put(0.11150802,0.26454615){\makebox(0,0)[lt]{\lineheight{1.25}\smash{\begin{tabular}[t]{l}$y_{i+1}$\end{tabular}}}}%
    \put(0,0){\includegraphics[width=\unitlength,page=3]{G0_v2.pdf}}%
    \put(0.35186263,0.06734282){\makebox(0,0)[lt]{\lineheight{1.25}\smash{\begin{tabular}[t]{l}$T$\end{tabular}}}}%
  \end{picture}%
\endgroup%

%\end{figure}
\end{minipage}
}
\subfigure[$G_1$]{
\begin{minipage}[t]{0.225\textwidth}
\centering
%% Creator: Inkscape 1.2 (dc2aedaf03, 2022-05-15), www.inkscape.org
%% Accompanies image file 'G1_v2.pdf' (pdf, eps, ps)
\begingroup%
  \makeatletter%
  \providecommand\color[2][]{%
    \errmessage{(Inkscape) Color is used for the text in Inkscape, but the package 'color.sty' is not loaded}%
    \renewcommand\color[2][]{}%
  }%
  \providecommand\transparent[1]{%
    \errmessage{(Inkscape) Transparency is used (non-zero) for the text in Inkscape, but the package 'transparent.sty' is not loaded}%
    \renewcommand\transparent[1]{}%
  }%
  \providecommand\rotatebox[2]{#2}%
  \newcommand*\fsize{\dimexpr\f@size pt\relax}%
  \newcommand*\lineheight[1]{\fontsize{\fsize}{#1\fsize}\selectfont}%
  \ifx\svgwidth\undefined%
    \setlength{\unitlength}{91.08550978bp}%
    \ifx\svgscale\undefined%
      \relax%
    \else%
      \setlength{\unitlength}{\unitlength * \real{\svgscale}}%
    \fi%
  \else%
    \setlength{\unitlength}{\svgwidth}%
  \fi%
  \global\let\svgwidth\undefined%
  \global\let\svgscale\undefined%
  \makeatother%
  \begin{picture}(1,1.05928159)%
    \lineheight{1}%
    \setlength\tabcolsep{0pt}%
    \put(0.50182419,1.02767842){\makebox(0,0)[lt]{\lineheight{1.25}\smash{\begin{tabular}[t]{l}$x$\end{tabular}}}}%
    \put(0,0){\includegraphics[width=\unitlength,page=1]{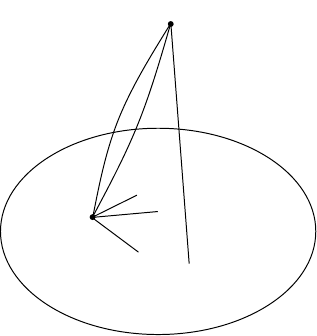}}%
    \put(0.11372328,0.33188428){\makebox(0,0)[lt]{\lineheight{1.25}\smash{\begin{tabular}[t]{l}$w_{T}$\end{tabular}}}}%
    \put(0,0){\includegraphics[width=\unitlength,page=2]{G1_v2.pdf}}%
  \end{picture}%
\endgroup%

%\end{figure}
\end{minipage}
}
\subfigure[$G_2$]{
\begin{minipage}[t]{0.225\textwidth}
\centering
%% Creator: Inkscape 1.2 (dc2aedaf03, 2022-05-15), www.inkscape.org
%% Accompanies image file 'G2_v2.pdf' (pdf, eps, ps)
\begingroup%
  \makeatletter%
  \providecommand\color[2][]{%
    \errmessage{(Inkscape) Color is used for the text in Inkscape, but the package 'color.sty' is not loaded}%
    \renewcommand\color[2][]{}%
  }%
  \providecommand\transparent[1]{%
    \errmessage{(Inkscape) Transparency is used (non-zero) for the text in Inkscape, but the package 'transparent.sty' is not loaded}%
    \renewcommand\transparent[1]{}%
  }%
  \providecommand\rotatebox[2]{#2}%
  \newcommand*\fsize{\dimexpr\f@size pt\relax}%
  \newcommand*\lineheight[1]{\fontsize{\fsize}{#1\fsize}\selectfont}%
  \ifx\svgwidth\undefined%
    \setlength{\unitlength}{91.08550978bp}%
    \ifx\svgscale\undefined%
      \relax%
    \else%
      \setlength{\unitlength}{\unitlength * \real{\svgscale}}%
    \fi%
  \else%
    \setlength{\unitlength}{\svgwidth}%
  \fi%
  \global\let\svgwidth\undefined%
  \global\let\svgscale\undefined%
  \makeatother%
  \begin{picture}(1,1.05928147)%
    \lineheight{1}%
    \setlength\tabcolsep{0pt}%
    \put(0.50182395,1.0276783){\makebox(0,0)[lt]{\lineheight{1.25}\smash{\begin{tabular}[t]{l}$x$\end{tabular}}}}%
    \put(0,0){\includegraphics[width=\unitlength,page=1]{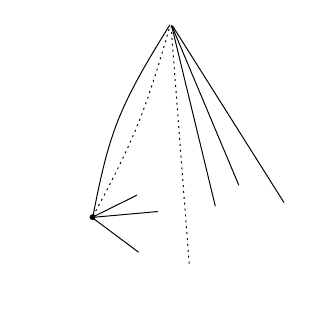}}%
    \put(0.11362381,0.33208638){\makebox(0,0)[lt]{\lineheight{1.25}\smash{\begin{tabular}[t]{l}$w_{T}$\end{tabular}}}}%
    \put(0,0){\includegraphics[width=\unitlength,page=2]{G2_v2.pdf}}%
    \put(0.6165894,0.27189538){\makebox(0,0)[lt]{\lineheight{1.25}\smash{\begin{tabular}[t]{l}$e_{j+1}$\end{tabular}}}}%
    \put(0.41546524,0.54709114){\makebox(0,0)[lt]{\lineheight{1.25}\smash{\begin{tabular}[t]{l}$e_j$\end{tabular}}}}%
  \end{picture}%
\endgroup%

%\end{figure}
\end{minipage}
}
\subfigure[$G_3$]{
\begin{minipage}[t]{0.225\textwidth}
\centering
%% Creator: Inkscape 1.2 (dc2aedaf03, 2022-05-15), www.inkscape.org
%% Accompanies image file 'G3_v2.pdf' (pdf, eps, ps)
\begingroup%
  \makeatletter%
  \providecommand\color[2][]{%
    \errmessage{(Inkscape) Color is used for the text in Inkscape, but the package 'color.sty' is not loaded}%
    \renewcommand\color[2][]{}%
  }%
  \providecommand\transparent[1]{%
    \errmessage{(Inkscape) Transparency is used (non-zero) for the text in Inkscape, but the package 'transparent.sty' is not loaded}%
    \renewcommand\transparent[1]{}%
  }%
  \providecommand\rotatebox[2]{#2}%
  \newcommand*\fsize{\dimexpr\f@size pt\relax}%
  \newcommand*\lineheight[1]{\fontsize{\fsize}{#1\fsize}\selectfont}%
  \ifx\svgwidth\undefined%
    \setlength{\unitlength}{91.08550978bp}%
    \ifx\svgscale\undefined%
      \relax%
    \else%
      \setlength{\unitlength}{\unitlength * \real{\svgscale}}%
    \fi%
  \else%
    \setlength{\unitlength}{\svgwidth}%
  \fi%
  \global\let\svgwidth\undefined%
  \global\let\svgscale\undefined%
  \makeatother%
  \begin{picture}(1,1.05002541)%
    \lineheight{1}%
    \setlength\tabcolsep{0pt}%
    \put(0,0){\includegraphics[width=\unitlength,page=1]{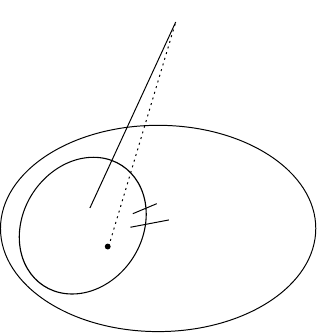}}%
    \put(0.15588433,0.38345075){\makebox(0,0)[lt]{\lineheight{1.25}\smash{\begin{tabular}[t]{l}$y_i$\end{tabular}}}}%
    \put(0.11118512,0.26432937){\makebox(0,0)[lt]{\lineheight{1.25}\smash{\begin{tabular}[t]{l}$y_{i+1}$\end{tabular}}}}%
    \put(0.51855912,1.01842225){\makebox(0,0)[lt]{\lineheight{1.25}\smash{\begin{tabular}[t]{l}$x$\end{tabular}}}}%
    \put(0,0){\includegraphics[width=\unitlength,page=2]{G3_v2.pdf}}%
    \put(0.35153972,0.06712605){\makebox(0,0)[lt]{\lineheight{1.25}\smash{\begin{tabular}[t]{l}$T$\end{tabular}}}}%
    \put(0,0){\includegraphics[width=\unitlength,page=3]{G3_v2.pdf}}%
    \put(0.62989914,0.26396847){\makebox(0,0)[lt]{\lineheight{1.25}\smash{\begin{tabular}[t]{l}$e_{j+1}$\end{tabular}}}}%
    \put(0.44524292,0.53916423){\makebox(0,0)[lt]{\lineheight{1.25}\smash{\begin{tabular}[t]{l}$e_j$\end{tabular}}}}%
  \end{picture}%
\endgroup%

%\end{figure}
\end{minipage}
}
\caption{An example for the graphs $G, G_1, G_2$ and $G_3$.}
\label{fig:lifting_Lemma}
\end{figure}
	
	Note that $V(G)=V(G_3)$ and $V(G_2)\setminus \{w_{T}\} = V(G_3) \setminus T$. Furthermore, we observe the following:
	\begin{itemize}
		\item for every $X \subseteq T$: $\vert \partial_{G}(X) \vert =  \vert \partial_{G_3}(X) \vert$,
		\item for every $X \subseteq V(G_2)\setminus \{w_{T}\}$: $\vert \partial_{G_2}(X) \vert = \vert\partial_{G_3}(X)\vert$ and $\vert\partial_{G_2}(X \cup \{w_{T}\})\vert=\vert\partial_{G_3}(X \cup T)\vert$.
	\end{itemize}
	
	Now, let $S \subseteq V(G_3)\setminus \{x\}$ be a set of odd cardinality. Set $A = S \cap T$ and $B=S \setminus A$. We consider two cases.
	
	{\bf Case 1.}\ $\vert A \vert$ is even.
	
	As a consequence, $B$ and $T \setminus A$ are sets of odd cardinality. Therefore, by using the above observations we obtain the following:
	\begin{align*}
	\vert \partial_{G_3}(S) \vert = \vert \partial_{G_3} (S^c) \vert
	&\geq \vert \partial_{G_3}(S^c\cap T) \vert + \vert \partial_{G_3}(S^c \cup T) \vert - \vert \partial_{G_3}(T) \vert\\
%	&= \vert \partial_{G_3}(T \setminus A) \vert + \vert \partial_{G_3}(B^c) \vert - \vert \partial_{G_3}(T) \vert\\
	&= \vert \partial_{G_3}(T \setminus A) \vert + \vert \partial_{G_3}(B) \vert - \vert \partial_{G_3}(T) \vert\\
	&=\vert \partial_{G}(T \setminus A) \vert + \vert \partial_{G_2}(B) \vert - \vert \partial_{G}(T) \vert\\
	&\geq r + r - r\\
	&=r.
	\end{align*}

	{\bf Case 2.}\ $\vert A \vert$ is odd.
	
	Thus, $B$ is a set of even cardinality, which implies
	\begin{align*}
	\vert \partial_{G_3}(S) \vert
	&\geq \vert \partial_{G_3}(S\cap T) \vert + \vert \partial_{G_3}(S \cup T) \vert - \vert \partial_{G_3}(T) \vert\\
	&= \vert \partial_{G_3}(A) \vert + \vert \partial_{G_3}(B \cup T) \vert - \vert \partial_{G_3}(T) \vert\\
	&=\vert \partial_{G}(A) \vert + \vert \partial_{G_2}(B \cup \{w_{T}\}) \vert - \vert \partial_{G}(T) \vert\\
	&\geq r + r - r\\
	&=r.
	\end{align*}
	In any case, we have $\vert \partial_{G_3}(S) \vert \geq r$, which implies $ \vert \partial_{G(e_j,e_{j+1})}(S') \vert\geq r $ for every $ S'\subseteq V(G(e_j,e_{j+1}))\setminus\{x\}$ of odd cardinality. This is a  contradiction to the assumption that $G$ is a counterexample.	
\end{proof}

The previous lemma can be used in $r$-graphs as follows.

\begin{theo}
	\label{theo:r-graph_lifting}
	Let $r\geq 2$ be an integer, let $G$ be a connected $r$-graph and let $X$ be a non-empty proper subset of $V(G)$. If $\vert X \vert$ is even, then $G/X$ can be transformed into a connected $r$-graph by  applying $\frac{1}{2}\left \vert \partial_G(X)\right|$ lifting operations at $w_X$. If $\vert X \vert$ is odd, then $G/X$ can be transformed into a connected $r$-graph by applying $\frac{1}{2}\left( \vert \partial_G(X) \vert - r \right)$ lifting operations at $w_X$.
%	Moreover, if $G/X$ is a connected graph, a planar graph or a connected planar graph, then this can be done such that the resulting graph is a connected graph, a planar graph or a connected planar graph, respectively.
\end{theo}

\begin{proof}

%If $G/X$ is planar, fix an embedding of $G$ and label the edges of $\partial_{G/X}(w_X)$ clockwise with respect to this embedding. 

Consider any labeling of $\partial_{G/X}(w_X)$. The statement follows by applying repeatedly Lemma \ref{Lem:r-graph lifting} to  $G/X$ at $w_X$. Note that $w_X$ is removed in the last step when $\vert X \vert$ is even.
\end{proof}

Note that the previous lifting operations can be applied so that they preserve embeddings of graphs in surfaces.

\subsection{Characterization of $\ca H_r$}

Let $ f$ be an $ H $-coloring of $ G $.
The subgraph of $H$ induced by the edge set $Im(f)$ is denoted by $H_f$. Observe that $H_f$ also colors $G$. Furthermore, if $H$ has no two vertices $u_1,u_2$ with $\partial_H(u_1)=\partial_H(u_2)$, then $f$ induces a mapping $f_V\colon V(G) \to V(H)$, where every $v \in V(G)$ is mapped to the unique vertex $u \in V(H)$ with $f(\partial_G(v))=\partial_H(u)$. Note that $f_V$ is well defined if $H$ is a connected graph with $|V(H)|>2$. A vertex of $V(H)\setminus Im(f_V)$ is called \emph{unused}.
%We will use the following lemma, that generalizes Lemma 2.4 of \cite{MTZ_r_graphs}.

%\begin{lem}[Mazzuoccolo et al. \cite{MTZ_r_graphs}]\label{lem:matching_covering_Im(f_V)}
%	Let $G$ and $H$ be graphs and let $f$ be an $H$-coloring of $G$ such that $f_V$ is well defined. If $M$ is a matching of $H$ that covers every vertex in $Im(f_V)$, then $f^{-1}(M)$ is a perfect matching of $G$.
%\end{lem}

%\begin{lem}\label{lem:matching_covering_Im(f_V)}
%	Let $G$ and $H$ be graphs and let $f$ be an $H$-coloring of $G$ such that $f_V$ is well defined. Moreover, let $M\subseteq E(H)$ be an edge set such that for all $u\in Im(f_V)$, $|\partial_H(u)\cap M|=1$. Then $f^{-1}(M)$ is a perfect matching of $G$.
%\end{lem}
%
%\begin{proof}
%	Let $v\in V(G)$ and let $u \in V(H)$ such that $f_V(v)=u$. Since $|\partial_H(u)\cap M|=1$ it follows that $|\partial_G(v)\cap f^{-1}(M)|=1.$ By the arbitrary choice of $v$ it follows that $f^{-1}(M)$ is a perfect matching of $G$.
%%	Since every vertex $v$ of $Im(f_V)$ is such that $|\partial_H(v)\cap M|=1$, it follows that $f^{-1}(M)$ is a perfect matching of $G$.
%\end{proof}

\begin{theo}
\label{theo:coloring_graphs_in_S(r,k)_generalisation}
Let $r \geq 3$ and let $G$ be an $r$-graph of class 2 that cannot be colored by an $r$-graph of smaller order. If $H$ is a connected $r$-graph and $f$ is an $H$-coloring of $G$, then $(f_V,f)$ is an isomorphism, i.e. $H \cong G$.
\end{theo}

\begin{proof}
Let $f\colon E(G) \to E(H)$ be an $H$-coloring of $G$. Note, that since $G$ is class 2, $H$ is also class 2 and therefore, $f_V$ is well defined. 
%In particular, $H$ is not isomorphic to the $r$-graph consisting of two vertices connected by $r$ edges, which means that $f_V$ is well defined. 
We first prove three claims.

\begin{claim}\label{claim:f_injective}
		$f$ is injective.
	\end{claim}
	
\emph{Proof of Claim \ref{claim:f_injective}.}
Suppose to the contrary that $f$ is not injective, which implies $|E(H_f)|<|E(G)|$. If $H$ contains no unused vertices, then $|E(H)|=|E(H_f)|<|E(G)|$, which contradicts the assumption that $G$ cannot be colored by an $r$-graph of smaller order. Thus, $H$ contains unused vertices; let $U \subseteq V(H)$ be the set of them. Transform the graph $H/U$ into a new $r$-graph $H'$ as follows. If $\vert U \vert$ is even, then apply $\frac{1}{2}\left \vert \partial_H(U)\right|$ lifting operations at $w_U$ (see Figure~\ref{fig:U_even}). If $\vert U \vert$ is odd, then apply $\frac{1}{2}\left( \vert \partial_H(U) \vert - r \right)$ lifting operations at $w_U$ (see Figure~\ref{fig:U_odd}). By Theorem~\ref{theo:r-graph_lifting}, this can be done in such a way that the resulting graph $H'$ is indeed an $r$-graph.

\begin{figure}[htbp]
\centering
\subfigure[$H$]{
\begin{minipage}[t]{0.3\textwidth}
\centering
%% Creator: Inkscape 1.2 (dc2aedaf03, 2022-05-15), www.inkscape.org
%% Accompanies image file 'U_even_1.pdf' (pdf, eps, ps)
\begingroup%
  \makeatletter%
  \providecommand\color[2][]{%
    \errmessage{(Inkscape) Color is used for the text in Inkscape, but the package 'color.sty' is not loaded}%
    \renewcommand\color[2][]{}%
  }%
  \providecommand\transparent[1]{%
    \errmessage{(Inkscape) Transparency is used (non-zero) for the text in Inkscape, but the package 'transparent.sty' is not loaded}%
    \renewcommand\transparent[1]{}%
  }%
  \providecommand\rotatebox[2]{#2}%
  \newcommand*\fsize{\dimexpr\f@size pt\relax}%
  \newcommand*\lineheight[1]{\fontsize{\fsize}{#1\fsize}\selectfont}%
  \ifx\svgwidth\undefined%
    \setlength{\unitlength}{83.05914758bp}%
    \ifx\svgscale\undefined%
      \relax%
    \else%
      \setlength{\unitlength}{\unitlength * \real{\svgscale}}%
    \fi%
  \else%
    \setlength{\unitlength}{\svgwidth}%
  \fi%
  \global\let\svgwidth\undefined%
  \global\let\svgscale\undefined%
  \makeatother%
  \begin{picture}(1,0.89001817)%
    \lineheight{1}%
    \setlength\tabcolsep{0pt}%
    \put(0,0){\includegraphics[width=\unitlength,page=1]{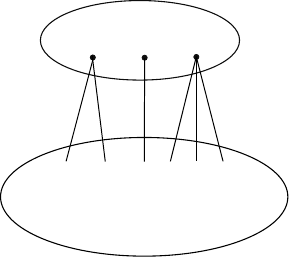}}%
    \put(0.39979654,0.75729389){\color[rgb]{0,0,0}\makebox(0,0)[lt]{\lineheight{0.46875152}\smash{\begin{tabular}[t]{l}$U$\end{tabular}}}}%
    \put(0.27832737,0.14507484){\color[rgb]{0,0,0}\makebox(0,0)[lt]{\lineheight{0.46875152}\smash{\begin{tabular}[t]{l}$Im(f_V)$\end{tabular}}}}%
  \end{picture}%
\endgroup%

%\end{figure}
\end{minipage}
}
\subfigure[$H/U$]{
\begin{minipage}[t]{0.3\textwidth}
\centering
%% Creator: Inkscape 1.2 (dc2aedaf03, 2022-05-15), www.inkscape.org
%% Accompanies image file 'U_even_2.pdf' (pdf, eps, ps)
\begingroup%
  \makeatletter%
  \providecommand\color[2][]{%
    \errmessage{(Inkscape) Color is used for the text in Inkscape, but the package 'color.sty' is not loaded}%
    \renewcommand\color[2][]{}%
  }%
  \providecommand\transparent[1]{%
    \errmessage{(Inkscape) Transparency is used (non-zero) for the text in Inkscape, but the package 'transparent.sty' is not loaded}%
    \renewcommand\transparent[1]{}%
  }%
  \providecommand\rotatebox[2]{#2}%
  \newcommand*\fsize{\dimexpr\f@size pt\relax}%
  \newcommand*\lineheight[1]{\fontsize{\fsize}{#1\fsize}\selectfont}%
  \ifx\svgwidth\undefined%
    \setlength{\unitlength}{83.05914758bp}%
    \ifx\svgscale\undefined%
      \relax%
    \else%
      \setlength{\unitlength}{\unitlength * \real{\svgscale}}%
    \fi%
  \else%
    \setlength{\unitlength}{\svgwidth}%
  \fi%
  \global\let\svgwidth\undefined%
  \global\let\svgscale\undefined%
  \makeatother%
  \begin{picture}(1,0.95159342)%
    \lineheight{1}%
    \setlength\tabcolsep{0pt}%
    \put(0,0){\includegraphics[width=\unitlength,page=1]{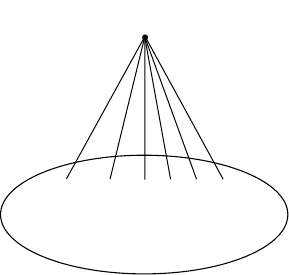}}%
    \put(0.39313846,0.8784152){\color[rgb]{0,0,0}\makebox(0,0)[lt]{\lineheight{0.46875152}\smash{\begin{tabular}[t]{l}$w_U$\end{tabular}}}}%
    \put(0.27832737,0.14507484){\color[rgb]{0,0,0}\makebox(0,0)[lt]{\lineheight{0.46875152}\smash{\begin{tabular}[t]{l}$Im(f_V)$\end{tabular}}}}%
  \end{picture}%
\endgroup%

\end{minipage}
}
\subfigure[$H'$]{
\begin{minipage}[t]{0.3\textwidth}
\centering
%% Creator: Inkscape 1.2 (dc2aedaf03, 2022-05-15), www.inkscape.org
%% Accompanies image file 'U_even_3.pdf' (pdf, eps, ps)
\begingroup%
  \makeatletter%
  \providecommand\color[2][]{%
    \errmessage{(Inkscape) Color is used for the text in Inkscape, but the package 'color.sty' is not loaded}%
    \renewcommand\color[2][]{}%
  }%
  \providecommand\transparent[1]{%
    \errmessage{(Inkscape) Transparency is used (non-zero) for the text in Inkscape, but the package 'transparent.sty' is not loaded}%
    \renewcommand\transparent[1]{}%
  }%
  \providecommand\rotatebox[2]{#2}%
  \newcommand*\fsize{\dimexpr\f@size pt\relax}%
  \newcommand*\lineheight[1]{\fontsize{\fsize}{#1\fsize}\selectfont}%
  \ifx\svgwidth\undefined%
    \setlength{\unitlength}{83.05914758bp}%
    \ifx\svgscale\undefined%
      \relax%
    \else%
      \setlength{\unitlength}{\unitlength * \real{\svgscale}}%
    \fi%
  \else%
    \setlength{\unitlength}{\svgwidth}%
  \fi%
  \global\let\svgwidth\undefined%
  \global\let\svgscale\undefined%
  \makeatother%
  \begin{picture}(1,0.48093009)%
    \lineheight{1}%
    \setlength\tabcolsep{0pt}%
    \put(0,0){\includegraphics[width=\unitlength,page=1]{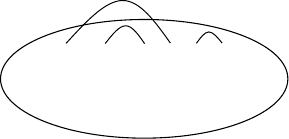}}%
    \put(0.27832737,0.14507484){\color[rgb]{0,0,0}\makebox(0,0)[lt]{\lineheight{0.46875152}\smash{\begin{tabular}[t]{l}$Im(f_V)$\end{tabular}}}}%
  \end{picture}%
\endgroup%

\end{minipage}
}
\caption{An example for the graphs $H$, $H/U$ and $H'$ when $|U|$ is even.}
\label{fig:U_even}
\end{figure}
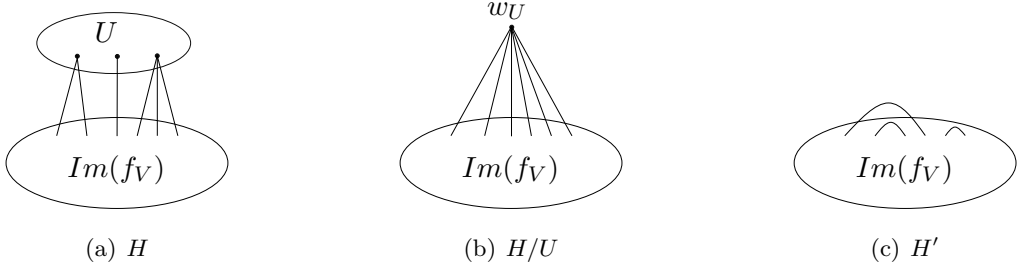

\begin{figure}[htbp]
\centering
\subfigure[$H$]{
\begin{minipage}[t]{0.3\textwidth}
\centering
%% Creator: Inkscape 1.2 (dc2aedaf03, 2022-05-15), www.inkscape.org
%% Accompanies image file 'U_odd_1.pdf' (pdf, eps, ps)
\begingroup%
  \makeatletter%
  \providecommand\color[2][]{%
    \errmessage{(Inkscape) Color is used for the text in Inkscape, but the package 'color.sty' is not loaded}%
    \renewcommand\color[2][]{}%
  }%
  \providecommand\transparent[1]{%
    \errmessage{(Inkscape) Transparency is used (non-zero) for the text in Inkscape, but the package 'transparent.sty' is not loaded}%
    \renewcommand\transparent[1]{}%
  }%
  \providecommand\rotatebox[2]{#2}%
  \newcommand*\fsize{\dimexpr\f@size pt\relax}%
  \newcommand*\lineheight[1]{\fontsize{\fsize}{#1\fsize}\selectfont}%
  \ifx\svgwidth\undefined%
    \setlength{\unitlength}{83.05914758bp}%
    \ifx\svgscale\undefined%
      \relax%
    \else%
      \setlength{\unitlength}{\unitlength * \real{\svgscale}}%
    \fi%
  \else%
    \setlength{\unitlength}{\svgwidth}%
  \fi%
  \global\let\svgwidth\undefined%
  \global\let\svgscale\undefined%
  \makeatother%
  \begin{picture}(1,0.89001817)%
    \lineheight{1}%
    \setlength\tabcolsep{0pt}%
    \put(0,0){\includegraphics[width=\unitlength,page=1]{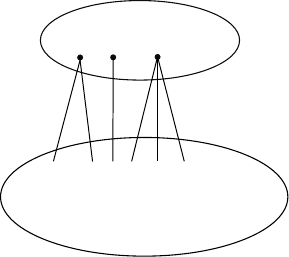}}%
    \put(0.39979654,0.75729389){\color[rgb]{0,0,0}\makebox(0,0)[lt]{\lineheight{0.46875152}\smash{\begin{tabular}[t]{l}$U$\end{tabular}}}}%
    \put(0.27832737,0.14507484){\color[rgb]{0,0,0}\makebox(0,0)[lt]{\lineheight{0.46875152}\smash{\begin{tabular}[t]{l}$Im(f_V)$\end{tabular}}}}%
    \put(0,0){\includegraphics[width=\unitlength,page=2]{U_odd_1.pdf}}%
  \end{picture}%
\endgroup%

%\end{figure}
\end{minipage}
}
\subfigure[$H/U$]{
\begin{minipage}[t]{0.3\textwidth}
\centering
%% Creator: Inkscape 1.2 (dc2aedaf03, 2022-05-15), www.inkscape.org
%% Accompanies image file 'U_odd_2.pdf' (pdf, eps, ps)
\begingroup%
  \makeatletter%
  \providecommand\color[2][]{%
    \errmessage{(Inkscape) Color is used for the text in Inkscape, but the package 'color.sty' is not loaded}%
    \renewcommand\color[2][]{}%
  }%
  \providecommand\transparent[1]{%
    \errmessage{(Inkscape) Transparency is used (non-zero) for the text in Inkscape, but the package 'transparent.sty' is not loaded}%
    \renewcommand\transparent[1]{}%
  }%
  \providecommand\rotatebox[2]{#2}%
  \newcommand*\fsize{\dimexpr\f@size pt\relax}%
  \newcommand*\lineheight[1]{\fontsize{\fsize}{#1\fsize}\selectfont}%
  \ifx\svgwidth\undefined%
    \setlength{\unitlength}{83.05914758bp}%
    \ifx\svgscale\undefined%
      \relax%
    \else%
      \setlength{\unitlength}{\unitlength * \real{\svgscale}}%
    \fi%
  \else%
    \setlength{\unitlength}{\svgwidth}%
  \fi%
  \global\let\svgwidth\undefined%
  \global\let\svgscale\undefined%
  \makeatother%
  \begin{picture}(1,0.83047211)%
    \lineheight{1}%
    \setlength\tabcolsep{0pt}%
    \put(0,0){\includegraphics[width=\unitlength,page=1]{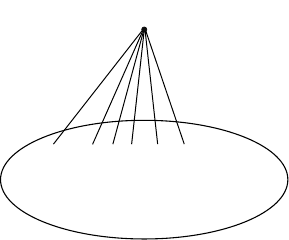}}%
    \put(0.39979654,0.75729389){\color[rgb]{0,0,0}\makebox(0,0)[lt]{\lineheight{0.46875152}\smash{\begin{tabular}[t]{l}$w_U$\end{tabular}}}}%
    \put(0.27832737,0.14507484){\color[rgb]{0,0,0}\makebox(0,0)[lt]{\lineheight{0.46875152}\smash{\begin{tabular}[t]{l}$Im(f_V)$\end{tabular}}}}%
    \put(0,0){\includegraphics[width=\unitlength,page=2]{U_odd_2.pdf}}%
  \end{picture}%
\endgroup%

\end{minipage}
}
\subfigure[$H'$]{
\begin{minipage}[t]{0.3\textwidth}
\centering
%% Creator: Inkscape 1.2 (dc2aedaf03, 2022-05-15), www.inkscape.org
%% Accompanies image file 'U_odd_3.pdf' (pdf, eps, ps)
\begingroup%
  \makeatletter%
  \providecommand\color[2][]{%
    \errmessage{(Inkscape) Color is used for the text in Inkscape, but the package 'color.sty' is not loaded}%
    \renewcommand\color[2][]{}%
  }%
  \providecommand\transparent[1]{%
    \errmessage{(Inkscape) Transparency is used (non-zero) for the text in Inkscape, but the package 'transparent.sty' is not loaded}%
    \renewcommand\transparent[1]{}%
  }%
  \providecommand\rotatebox[2]{#2}%
  \newcommand*\fsize{\dimexpr\f@size pt\relax}%
  \newcommand*\lineheight[1]{\fontsize{\fsize}{#1\fsize}\selectfont}%
  \ifx\svgwidth\undefined%
    \setlength{\unitlength}{83.05914758bp}%
    \ifx\svgscale\undefined%
      \relax%
    \else%
      \setlength{\unitlength}{\unitlength * \real{\svgscale}}%
    \fi%
  \else%
    \setlength{\unitlength}{\svgwidth}%
  \fi%
  \global\let\svgwidth\undefined%
  \global\let\svgscale\undefined%
  \makeatother%
  \begin{picture}(1,0.85016072)%
    \lineheight{1}%
    \setlength\tabcolsep{0pt}%
    \put(0,0){\includegraphics[width=\unitlength,page=1]{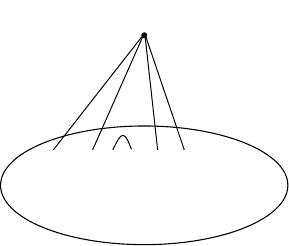}}%
    \put(0.3535282,0.7769825){\color[rgb]{0,0,0}\makebox(0,0)[lt]{\lineheight{0.46875152}\smash{\begin{tabular}[t]{l}$w_U$\end{tabular}}}}%
    \put(0.27832737,0.14507484){\color[rgb]{0,0,0}\makebox(0,0)[lt]{\lineheight{0.46875152}\smash{\begin{tabular}[t]{l}$Im(f_V)$\end{tabular}}}}%
    \put(0,0){\includegraphics[width=\unitlength,page=2]{U_odd_3.pdf}}%
  \end{picture}%
\endgroup%

\end{minipage}
}
\caption{An example for the graphs $H$, $H/U$ and $H'$ when $|U|$ is odd.}
\label{fig:U_odd}
\end{figure}

	Note that every edge of $Im(f)$ is incident with at most one vertex of $U$. Thus, we can define a function $f' \colon E(G) \to E(H/U)$ as follows. For every $e \in E(G)$ let $f'(e)$ be the edge of $H/U$ corresponding to the edge $f(e)$ of $H$. Observe that $f'$ is an $H/U$-coloring of $G$, where $w_U$ is the only unused vertex. Next, define a new mapping $f'' \colon E(G) \to E(H')$ as follows. For every $e \in E(G)$ set 
	$$f''(e)=\begin{cases}
  e'  & \text{if } f'(e) \text{ is one of the two edges lifted to } e', \\
  f'(e) & \text{if } f'(e) \in E(H').
\end{cases}$$

By construction, $f''(\partial_G(v))=\partial_{H'}(f_V(v))$ for every $v \in V(G)$.
Since $G$ and $H'$ are $r$-regular it follows that $f''$ is an $H'$-coloring. Therefore, $H' \prec G$ and hence $|V(H')| \geq |V(G)|$ by our assumptions. This is a contradiction, since
\begin{align*}
\vert E(H') \vert \leq \vert E(H/U) \vert = \vert E(H_f) \vert < \vert E(G) \vert.
\end{align*}
	\ENDproof

\begin{claim}\label{claim:f_v_surjective}
		$f_V$ is surjective.
	\end{claim}
	
	\emph{Proof of Claim \ref{claim:f_v_surjective}.}
Suppose that $H$ contains unused vertices. Then, there are $v_1,v_2 \in V(G)$ and $e \in [v_1,v_2]_{G}$ such that $f(e)$ is incident with exactly one unused vertex in $ H $, since $H$ is connected. Thus, $f(\partial_G(v_1))=f(\partial_G(v_2))$, which contradicts Claim~\ref{claim:f_injective}.
	\ENDproof

\begin{claim}\label{claim:order_H}
		$|V(H)| = |V(G)|$.
	\end{claim}
	
\emph{Proof of Claim \ref{claim:order_H}.}
Since $G$ cannot be colored by an $r$-graph of smaller order, we have $|V(H)| \geq |V(G)|$. On the other hand, $|V(H)| \leq |V(G)|$ by Claim~\ref{claim:f_v_surjective}.
	\ENDproof

	Claims \ref{claim:f_injective}, \ref{claim:f_v_surjective} and \ref{claim:order_H} imply that $f$ and $f_V$ are bijections. Furthermore,  we obtain that $e \in [v_1,v_2]_{G}$ if and only if $f(e) \in [f_V(v_1),f_V(v_2)]_H$. Therefore, $(f_V,f)$ is an isomorphism between $ G $ and $ H $, i.e. $H \cong G$.
\end{proof}

In \cite{Mkrtchyan_Pet_col}, Mkrtchyan proves that if a connected $3$-graph $H$ colors the Petersen graph $P$, then $H\cong P$. The following result is implied by Theorem~\ref{theo:coloring_graphs_in_S(r,k)_generalisation} together with Observation \ref{obs:coloring_basics} $(ii)$ and  gives a generalization of Mkrtchyan's result in the $r$-regular case.

\begin{cor}
\label{cor:coloring_graphs_in_S(r,k)}
Let $r \geq 3$ and let $G$ be an $r$-graph of class 2 such that $\pi(G')>\pi(G)$ for every $r$-graph $G'$ with $|V(G')|<|V(G)|$. If $H$ is a connected $r$-graph with $H \prec G$, then $H \cong G$.
\end{cor}

 By Theorem \ref{thm: lower bound S(r,r-2)}, $\ca T_r = \ca T(r,r-2) =\{P^{\ca M}\colon \ca M $ is a set of $r-3$ perfect matchings of the Petersen graph $ P\}$.
Hence, with Corollary~\ref{cor:coloring_graphs_in_S(r,k)} we obtain the following theorem.

\begin{theo}
\label{theo:coloring_P^M}
%Let $r\geq3$ be an integer, let $H$ be a connected $r$-graph and let $\mathcal{M}$ be a multiset of $r-3$ perfect matchings of $P$. If $H \prec P^{\mathcal{M}}$, then $H \cong P^{\mathcal{M}}$.

Let $r\geq3$, let $H$ be a connected $r$-graph and let $G\in \ca T(r,r-2)\cup \ca T(r,1)$. If $H \prec G$, then $H \cong G$.
\end{theo}

\begin{theo}
\label{theo:characterisation H_r}
	Let $r\ge3$ and let $G$ be a connected $r$-graph. The following statements are equivalent.
	\begin{itemize}
		\item[1)] $G\in\ca H_r$.
		\item[2)] The only connected $r$-graph coloring $G$ is $G$ itself.
		\item[3)] $G$ cannot be colored by a smaller $r$-graph.
	\end{itemize}
\end{theo}

\begin{proof}
	$2) \implies 1)$ follows trivially.
	
	$1) \implies 3)$. Assume by contradiction that $3)$ is not true. Then, let $H$ be a smallest $r$-graph smaller than $G$ such that $H\prec G$. Note that $H$ cannot be colored by a smaller $r$-graph because otherwise, since the relation $\prec$ is transitive, $G$ would be colored by an $r$-graph smaller than $H$. Hence, $H\in \ca H_r$ by Theorem \ref{theo:coloring_graphs_in_S(r,k)_generalisation}. Thus, $\ca H_r \setminus \{G\}$ is an $r$-complete set, in contradiction to the inclusion-wise minimality of $\ca H_r$.
	
	$3) \implies 2)$ follows by Theorem \ref{theo:coloring_graphs_in_S(r,k)_generalisation}.
\end{proof}

\begin{cor}
For every $r \geq 3$, there exists only one inclusion-wise minimal $r$-complete set, i.e. $\ca H_r$ is unique.
\label{cor:H_r unique}
\end{cor}

For $r=3$, we have $\ca T(r,r-2)=\ca T(r,1) = \{P\}$. The Petersen Coloring Conjecture states that $\ca H_3=\{P\}$. This situation is very exclusive as we show in the following subsection.

\subsection{Infinite subsets of $\ca H_r$}

%\begin{obs}
%	Let $r\ge 3$. For any multiset $\ca M$ of $r-3$ perfect matchings of $P$, the graph $  P^{\ca M}$ is $ 3 $-edge-connected.
%\end{obs}

\begin{lem}\label{Lemma-2-cut-image}
	Let $r \geq 3$, let $G$ and $H$ be two connected $ r $-graphs and let $f$ be an $H$-coloring of $G$. For any $ 2 $-edge-cut $ F=\{e_1,e_2\}\subseteq E(G) $, either $ |f(F)|=1 $ or $ f(F) $ is a $2$-edge-cut of $ H $.
\end{lem}

\begin{proof}
	Let $u$ and $v$ be the endvertices of $f(e_1)$.	Suppose by contradiction that $|f(F)|=2$ but $f(F)$ is not a $2$-edge-cut of $H$. Then, there is a $uv$-path $T$ in $H$ avoiding the edges of $f(F)$. Consider the circuit $C=T+f(e_1)$. By Observation~\ref{obs:coloring_basics} $(iii)$, $f^{-1}(E(C))$ is a union of circuits of $G$. This is a contradiction, since $f^{-1}(E(C))$ contains $e_1$ but not $e_2$.
\end{proof}

Let $G,H$ be two graphs, let $f \colon E(G) \to E(H)$, $g \colon V(G) \to V(H)$ and let $G'$ be a subgraph of $G$. The restriction of $f$ to $E(G')$ is denoted by $f|_{G'}$; the restriction of $g$ to $V(G')$ is denoted by $g|_{G'}$.

\begin{lem}\label{Lemma-PM-e-cong-itselt}
	Let $G$ and $H$ be two $ r $-graphs, where $r \geq 3$, and let $f$ be an $H$-coloring of $G$. 
	Let $\ca M$ be a multiset of $r-3 $ perfect matchings of  $P$ and let $ e_0\in E(P^{\ca M}) $. Let $ G' $ be an induced subgraph of $ G $ isomorphic to $P^{\ca M}-e_0$ and $ H' $ be the subgraph of $ H $ induced by $ f(E(G')) $. 
Then, $ (f_V|_{G'} , f|_{G'} ) $ is an isomorphism between $ G' $ and $ H' $, i.e.\ $ H'\cong G' $.  
\end{lem}

\begin{proof}
	By the definition of $ G' $, we have $ | \partial_G(V(G')) |=2$. Assume that $\partial_G(V(G')) =\{e_1,e_2\}$ and $ e_i=u_iv_i $ with $ u_i\in V(G') $ for each $ i\in \{1,2\} $.
	
	We first consider the case $ f(e_1)=f(e_2) $. Let $ G^* $ be the $ r $-graph obtained from $ G' $ by adding a new edge $ e_3 $ joining $ u_1 $ and $ u_2 $. Set $ f^*(e)= f(e)=f|_{G'}(e) $ for each $ e\in E(G^*) \setminus \{e_3\}$ and $ f^*(e_3)= f(e_1)=f(e_2) $. Then $ f^* $ is an $ H $-coloring of $ G^* $. Since $ G^* \cong P^{\ca M}$, we have that $ (f^*_V,f^*) $ is an isomorphism between $ G^* $ and $ H$ by Theorem \ref{theo:coloring_graphs_in_S(r,k)_generalisation}. Thus $ (f_V|_{G'} , f|_{G'} ) $ is an isomorphism of $ G' $ to $ H' $ by the definition of $ f^* $.
	
	Now we assume that $ f(e_1)\neq f(e_2) $. By Lemma \ref{Lemma-2-cut-image}, $\{f(e_1), f(e_2) \}$  is a $ 2 $-edge-cut of $ H $. Let  $ X $ be a subset of $ V(H) $ such that $ \partial_H(X)=\{f(e_1), f(e_2) \}$. Denote $ f(e_i)=x_iy_i$ with $ x_i\in X $ for each $ i\in\{1,2\} $. We consider the following two cases.
	
	{\bf Case 1.} $f_V(V(G')) \subseteq X$ or $f_V(V(G'))\subseteq V(H)\setminus X$.
	
Without loss of generality, assume that  $f_V(V(G'))\subseteq X$. Let $ G^* $ be the $ r $-graph obtained from $ G' $ by adding a new edge $ e_3 $ joining $ u_1 $ and $ u_2 $, and  $ H^* $ be the $ r $-graph obtained from $ H[X] $ by adding a new edge $ e_4 $ joining $ x_1 $ and $ x_2 $. Set  $ f^*(e)= f(e)=f|_{G'}(e) $ for each $ e\in E(G^*) \setminus \{e_3\}$ and $ f^*(e_3)= e_4 $. Then $ f^* $ is an $ H^* $-coloring of $ G^* $. Since $ G^* \cong P^{\ca M}$, we have that $ (f^*_V,f^*) $ is an isomorphism between $ G^* $ and $ H^*$ by Theorem \ref{theo:coloring_graphs_in_S(r,k)_generalisation}. Thus $ (f_V|_{G'} , f|_{G'} ) $ is an isomorphism of $ G' $ to $ H' $ by the definition of $ f^* $ and the statement follows.

{\bf Case 2.} $f_V(V(G'))\cap X\neq \emptyset~\text{and}~f_V(V(G'))\cap (V(H)\setminus X)\neq \emptyset$.

We show that this case does not apply. Let $Z_1=f_V(V(G'))\cap X$ and $Z_2=f_V(V(G'))\cap (V(H)\setminus X)$. Observe that $\{ f(e_1),f(e_2) \}\subseteq \partial_H(Z_1)\cup \partial_H(Z_2)$.
	Set $ U_1=X\setminus Z_1 $ and $ U_2= (V(H)\setminus X)\setminus Z_2$. Note that $ U_1$ and $ U_2$ might be empty. We construct a new $r$-graph $H_2$ from $H$ in two steps.
	First, if $ U_1=\emptyset$, set $ H_1=H $. Otherwise we can construct an $r$-graph $ H_1 $ starting from $ H/U_1 $ by taking suitable lifting operations at $w_{U_1}$ as described in Theorem \ref{theo:r-graph_lifting}, namely: if $\vert U_1 \vert$ is even, then apply $\frac{1}{2}\left \vert \partial_H(U_1)\right|$ lifting operations at $w_{U_1}$; if $\vert U_1 \vert$ is odd, then apply $\frac{1}{2}\left( \vert \partial_H(U_1) \vert - r \right)$ lifting operations at $w_{U_1}$.
%By Theorem~\ref{theo:r-graph_lifting} this can be done such that the resulting graph is indeed an $r$-graph.
Observe that $ U_2\subset V(H_1)$. Next, if $ U_2=\emptyset$, set $ H_2=H_1$. Otherwise let $ H_2 $ be a graph obtained from $ H_1/U_2 $ by taking similar lifting operations as described above at the vertex $w_{U_2}$. An example for the construction of $H_2$ is given in Figure~\ref{fig:H_H1_H2}.

\begin{figure}[htbp]
\centering
\subfigure[$H$]{
\begin{minipage}[t]{0.3\textwidth}
\centering
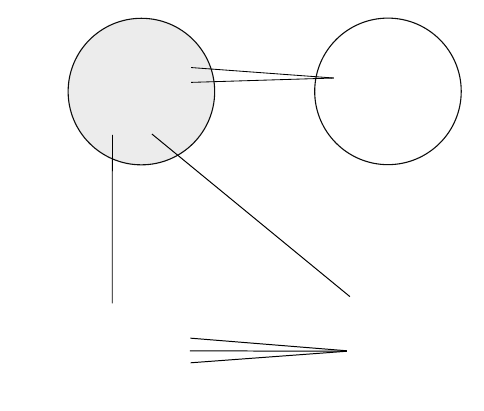
\end{minipage}
}
\subfigure[$H_1$]{
\begin{minipage}[t]{0.3\textwidth}
\centering
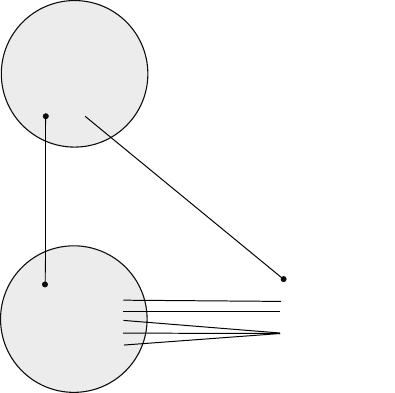
\end{minipage}
}
\subfigure[$H_2$]{
\begin{minipage}[t]{0.3\textwidth}
\centering
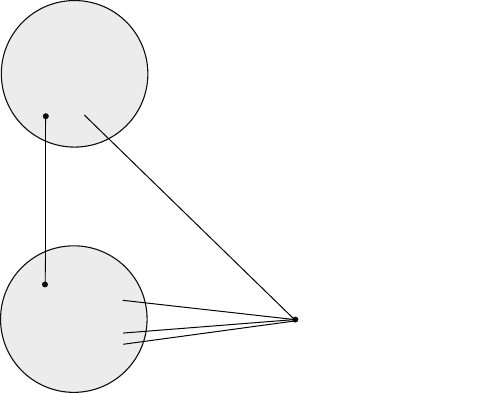
\end{minipage}
}
\caption{An example for the graphs $H$, $H_1$ and $H_2$, when $U_1, U_2$ are non-empty, $U_1$ is of even cardinality and $U_2$ is of odd cardinality.}
\label{fig:H_H1_H2}
\end{figure}

By Theorem~\ref{theo:r-graph_lifting}, this can be done such that $H_2$ is an $r$-graph. 
	Furthermore, we have $$|E(H_2)| \leq |E(H') \cup \{f(e_1), f(e_2)\}| \leq |E(G')|+2.$$
	As a consequence, $|V(H_2)|\leq 10$. Thus, $ H_2 $ is class $ 1 $ since it has a $ 2 $-edge-cut and hence, $ H_2 $ has $ r $ pairwise disjoint perfect matchings. By the construction of $ H_2 $, we deduce that $ H $ contains  $ r $ pairwise disjoint sets of edges, denoted by  $ S_1,\ldots,S_r $, such that  $|\partial_H(y)\cap S_j|=1$ for each $y\in f_V(V(G'))$  and each $ j\in\{1,\ldots,r\} $. Then $ f^{-1}(S_1), \ldots,  f^{-1}(S_r)$ are  $ r $ pairwise disjoint sets  of edges of $ G $ such that  $|\partial_G(u )\cap f^{-1}(S_j)|=1$ for each $u\in V(G')$  and each $ j\in\{1,\ldots,r\} $. This is a contradiction since $ G' $ is class $ 2 $.
\end{proof}

Let $ G $ and $ G' $ be two disjoint $ r $-graphs  of class $ 2 $  with $ e\in E(G) $ and $ e'\in E(G') $. Denote by $ (G,e)| (G',e')$  the set of all  graphs obtained from $ G $ by replacing the edge $ e $ of $ G $ by $ (G',e')$, that is,  deleting $ e $ from $ G $ and $ e' $ from $ G' $, and then adding two edges between $ V(G) $ and $ V(G') $ such that the resulting graph is regular (see Figure~\ref{fig:replacing_edge}).

\begin{figure}[htbp]
\centering
\scalebox{1}{%% Creator: Inkscape 1.2 (dc2aedaf03, 2022-05-15), www.inkscape.org
%% Accompanies image file 'replacing_edge.pdf' (pdf, eps, ps)
\begingroup%
  \makeatletter%
  \providecommand\color[2][]{%
    \errmessage{(Inkscape) Color is used for the text in Inkscape, but the package 'color.sty' is not loaded}%
    \renewcommand\color[2][]{}%
  }%
  \providecommand\transparent[1]{%
    \errmessage{(Inkscape) Transparency is used (non-zero) for the text in Inkscape, but the package 'transparent.sty' is not loaded}%
    \renewcommand\transparent[1]{}%
  }%
  \providecommand\rotatebox[2]{#2}%
  \newcommand*\fsize{\dimexpr\f@size pt\relax}%
  \newcommand*\lineheight[1]{\fontsize{\fsize}{#1\fsize}\selectfont}%
  \ifx\svgwidth\undefined%
    \setlength{\unitlength}{177.47439233bp}%
    \ifx\svgscale\undefined%
      \relax%
    \else%
      \setlength{\unitlength}{\unitlength * \real{\svgscale}}%
    \fi%
  \else%
    \setlength{\unitlength}{\svgwidth}%
  \fi%
  \global\let\svgwidth\undefined%
  \global\let\svgscale\undefined%
  \makeatother%
  \begin{picture}(1,0.19149986)%
    \lineheight{1}%
    \setlength\tabcolsep{0pt}%
    \put(0,0){\includegraphics[width=\unitlength,page=1]{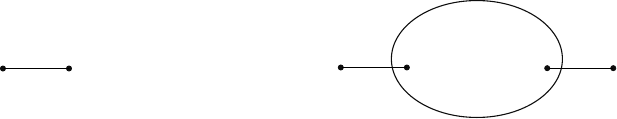}}%
    \put(0.03757375,0.04343616){\color[rgb]{0,0,0}\makebox(0,0)[lt]{\lineheight{0.46875152}\smash{\begin{tabular}[t]{l}$e$\end{tabular}}}}%
    \put(0.68543645,0.10734421){\color[rgb]{0,0,0}\makebox(0,0)[lt]{\lineheight{0.46875152}\smash{\begin{tabular}[t]{l}$G'-e'$\end{tabular}}}}%
    \put(0,0){\includegraphics[width=\unitlength,page=2]{replacing_edge.pdf}}%
  \end{picture}%
\endgroup%
} 
\caption{A replacement of the edge $e$ by $(G',e')$.}
\label{fig:replacing_edge}
\end{figure}

In fact, any graph in $ (G,e)| (G',e')$ is an $ r $-graph of class $ 2 $. Furthermore, we use $ G | (G',e')$ to denote the set of all  graphs obtained from $ G $ by replacing each edge of $ G $ by $ (G',e')$.

\begin{theo}
\label{theo:inductive construction}
	Let $\ca M$ be a multiset of $r-3 $ perfect matchings of $P$, where $r \geq 3$, and let $ e_0\in E(P^{\ca M}) $. Let $G$ be an $r$-graph such that $ G\ncong  P^{\ca M} $. If $ G \in \ca H_r $ , then $ G | (P^{\ca M},e_0) \subset \ca H_r$. 
\end{theo}
\begin{proof}
	By Theorem~\ref{theo:characterisation H_r}, it suffices to prove that any $ G^*\in G | (P^{\ca M},e_0)$ cannot be colored by a connected $r$-graph of smaller order. 
	Let $ H $ be a connected
	$r$-graph such that $ G^* $ has an $ H $-coloring, denoted by $ f $. Label all  subgraphs  of $ G^*$ isomorphic to $P^{\ca M}-e_0$ as $ G_1, \ldots, G_\ell $, where $ \ell=|E(G)| $, and denote by $ H_i $ the subgraph of $ H $ induced by $ f_V(V(G_i)) $. Note that $H_i \cong P^{\ca M}-e_0$ by Lemma~\ref{Lemma-PM-e-cong-itselt}. For each $ i\in \{1,\ldots,\ell\}$, we label the two edges of  $ \partial_{G^*}(V(G_i))  $ as $ e^1_i $ and $ e^2_i $, and let $ e^t_i=u^t_iv^t_i $ with $ v^t_i\notin  V(G_i) $ for each $ t\in\{1,2\} $.
	
	\begin{claim}\label{Claim-2-edge-cut-induced graph}
		$ f(\partial_{G^*}(V(G_i)) )$ is a $ 2 $-edge-cut in $ H $, for every $ i\in\{1,\ldots,\ell\} $. 
	\end{claim}
	
	\emph{Proof of Claim \ref{Claim-2-edge-cut-induced graph}.}
	By Lemma \ref{Lemma-2-cut-image}, we suppose to the contrary that there is $ i\in\{1,\ldots,\ell\} $ such that $ f(e^1_i)=f(e^2_i) $. With $ G_i\cong  P^{\ca M}-e_0$, we have $ H\prec P^{\ca M}  $ by Lemma \ref{Lemma-PM-e-cong-itselt}, and so $ H\cong P^{\ca M} $ by Theorem \ref{theo:coloring_graphs_in_S(r,k)_generalisation}. Then, $ |f(F)|=1$  for any $ 2 $-edge-cut $ F\subset E(G^*)$ by Lemma \ref{Lemma-2-cut-image} since $ P^{\ca M} $ is $ 3 $-edge-connected. Thus, by the construction of $ G^* $, we have $ H\prec G $, which implies $ H\cong G $ by Theorem \ref{theo:coloring_graphs_in_S(r,k)_generalisation}. This is a contradiction to the fact that $ G\ncong  P^{\ca M} $.
	\ENDproof
	
	\begin{claim}\label{Claim-PM-induced-graph-same-image}
	 $ V(H_i) = V(H_j) $ or $V(H_i) \cap V(H_j)  =\emptyset$, for every $i,j\in \{1,\ldots,\ell\}$.
%		For any two distinct indices $i,j\in \{1,\ldots,\ell\}$,  we have $ V(H_i) = V(H_j) $ or $V(H_i) \cap V(H_j)  =\emptyset$. 
	\end{claim}
	
	\emph{Proof of Claim \ref{Claim-PM-induced-graph-same-image}.}
	Assume $ V(H_i)  \cap V(H_j)  \neq\emptyset$. To complete the proof, we shall show $V(H_i) \setminus V(H_j) = \emptyset$  and $V(H_j) \setminus V(H_i) = \emptyset$.
	Without loss of generality, suppose to the contrary that $ V(H_j) \setminus V(H_i)  \neq\emptyset$. Note that $ f(\partial_{G^*}(V(G_i)) )$  is a $ 2 $-edge-cut in $ H $ by Claim \ref{Claim-2-edge-cut-induced graph}.
	%Let $ H_j $ be the subgraph of $ H $ induced by $ f_V(V(G_j)) $.
	Observe that both $H_i $ and $H_j $ are isomorphic to $P^{\ca M} - e_0 $ by Lemma \ref{Lemma-PM-e-cong-itselt}. %$H_j $ is connected since $ G_j $ is connected.
	Thus, at least one edge of  $ f(\partial_{G^*}(V(G_i)) )$ is contained in $ E(H_j) $, since $H_j $ is connected. As a consequence,  $ H_j $ either has a bridge or a $2$-edge-cut consisting of two non-adjacent edges, since an $r$-graph has no cut-vertex. This is not possible. %This is a contradiction to the fact that $ H_j\cong P^{\ca M}-e_0 $ by Lemma \ref{Lemma-PM-e-cong-itselt}. 
	\ENDproof
	
	\begin{claim}\label{Claim-vertex-not-in-induce-graph}
		$ f_V(z) \notin \bigcup^{\ell}_{i=1}V(H_i)$, for every $z\in V(G^*)\setminus(\bigcup^{\ell}_{i=1}V(G_i)) $.
	\end{claim}
	
	\emph{Proof of Claim \ref{Claim-vertex-not-in-induce-graph}.}
	Suppose to the contrary that there is a vertex $z\in V(G^*)\setminus(\bigcup^{\ell}_{i=1}V(G_i)) $ such that 
	$ f_V(z)\in V(H_j) $ for some $ j\in \{1,\ldots,\ell\}$. Let $ e $ be an edge incident with $ f_V(z)$ in $ H_j $.
	By the construction of $ G^* $, the only edge of $ f^{-1}(e)\cap \partial_{G^*}(z)$ is an element of $ \partial_{G^*}(V(G_k)) $ for some $ k\in\{1,\ldots,\ell\} $. Thus, $ e $ is   in a $ 2 $-edge-cut of $ H $ by Claim \ref{Claim-2-edge-cut-induced graph}, contradicting the fact that $ H_j\cong  P^{\ca M}-e_0$ by Lemma \ref{Lemma-PM-e-cong-itselt}. 
	\ENDproof
	
	  By Claim \ref{Claim-2-edge-cut-induced graph},  $\partial_{H}(V(H_i))=  f(\partial_{G^*}(V(G_i)) ) =\{f(e^1_i), f(e^2_i) \}$. Let $ f(e^t_i)=x^t_iy^t_i $ with $ y^t_i\notin  V(H_i) $ for each $ t\in\{1,2\} $.  
	
	\begin{claim}\label{Claim-y^t_i-notin-Hj}
		$\{ y^1_i,y^2_i\}\cap  V(H_j)=\emptyset$, for every $ i,j\in \{1,\ldots,\ell\}$.
	\end{claim}
	
	\emph{Proof of Claim \ref{Claim-y^t_i-notin-Hj}.}
	By contradiction, suppose $ y^t_i\in  V(H_j)$ for some $ t\in\{1,2\} $. Note that $ f_V(v^t_i)\in\{y^t_i, x^t_i \} $ and $ x^t_i\in  V(H_i)$.
	Thus, $ f_V(v^t_i)\in V(H_i)\cup V(H_j)$. This is a contradiction to Claim \ref{Claim-vertex-not-in-induce-graph} since $ v^t_i \in V(G^*)\setminus(\bigcup^{\ell}_{i=1}V(G_i)) $ by the construction of $ G^* $.
	\ENDproof
	
	%\begin{claim}\label{Claim-gi-neq-gj}
	%	For each $ i,j\in \{1,\ldots,\ell\}$, $ f(e^t_i)\neq f(e^s_j)$  for each $ t,s\in \{1,2\}$ if $ V(H_i) \neq  V(H_j) $. 
	%\end{claim}
	
	%\emph{Proof of Claim \ref{Claim-gi-neq-gj}.}
	%By contradiction, suppose $ f(e^t_i)=f(e^s_j)$  for some $ s,t\in \{1,2\}$. Note that $ f_V(u^t_i)\in\{x^t_i, y^t_i \} $ and $ y^t_i\in  V(H_i)$.
	%By Claim \ref{Claim-PM-induced-graph-same-image}, 
	%$V(H_i) \cap V(H_j)  =\emptyset$ and so $ x^t_i\in  V(H_j)$. It implies  $ f_V(u^t_i)\in V(H_i)\cup V(H_j)$. This is a contradiction to Claim \ref{Claim-vertex-not-in-induce-graph} since $ u^t_i \in V(G^*)\setminus(\bigcup^{\ell}_{i=1}V(G_i)) $ by the construction of $ G^* $.
	%\ENDproof

	Note that $G$ can be obtained from $G^*$ by deleting all vertices of $ G_i $ and adding a new edge edge $ e_ i$ joining $v^1_i$ and $v^2_i$ for each $ i\in \{1,\ldots,\ell\}$.
	By Claims \ref{Claim-PM-induced-graph-same-image} and \ref{Claim-y^t_i-notin-Hj},  $(V(H_i)\cup\{y^1_i,y^2_i\}) \cap V(H_j)  =\emptyset$ if $ V(H_i) \neq V(H_j) $ for  each $ i,j\in \{1,\ldots,\ell\}$. Thus, we can construct an $ r $-graph $ H' $  from $ H $ by deleting all vertices of $ H_i $ and adding a new edge $ g_ i$ joining $y^1_i$ and $y^2_i$ for each $ i\in \{1,\ldots,\ell\}$.  Note that, for some $i\ne j \in\{1,\dots,\ell\}$, it might happen that $ V(H_i)=V(H_j) $. In such a case, $g_i=g_j$.	Define a mapping $ f' \colon  E(G) \to E(H')$ by letting $ f'(e_i)=g_i$, for each $ i\in \{1,\ldots,\ell\}$.  By Claim \ref{Claim-vertex-not-in-induce-graph}, $ f'_V(z)\in V(H')$  for every vertex $ z  \in V(G)\subset V(G^*)$.
	Furthermore, we have $ f'(\partial_{G}(z)) =\partial_{H'}(f'_V(z))$. Since both $ G $ and $ H' $ are $ r $-graphs,  $ f'$ is proper. Thus, $ f'$ is an $ H' $-coloring of $ G$. Then, $( f'_V,  f')$ is an isomorphism between $G$ and $H'$ by Theorem \ref{theo:coloring_graphs_in_S(r,k)_generalisation}. This implies that $ |V(G^*)\setminus(\bigcup^{\ell}_{i=1}V(G_i))| =  |V(H)\setminus(\bigcup^{\ell}_{i=1}V(H_i))|  $, and $ V(H_i)\neq V(H_j)$ for any distinct $ i,j\in \{1,\ldots,\ell\}$ since $ f' (e_i)\neq  f' (e_j)$.
	Therefore, $ |V(G^*)|=|V(H)|$ by Claims~\ref{Claim-PM-induced-graph-same-image} and \ref{Claim-vertex-not-in-induce-graph}, which completes the proof.
\end{proof}

The following corollary answers the question of 
\cite{MTZ_r_graphs} whether for each $r \geq 4$, there exists a
connected $r$-graph $H$ with $H \prec G$ for every $r$-graph $G$.

\begin{cor}
	\label{cor:H_3 = P or infinite}
	Either $\ca H_3 = \{P\}$ or $\ca H_3$ is an infinite set.	
	Moreover, if $r\geq 4$, then $\ca H_r$ is an infinite set.
\end{cor}

\begin{proof}
	If $\ca H_3 \neq \{P\}$, then there is a smallest $3$-graph $G$ that cannot be colored by $P$. Note that $G$ is class 2 and not isomorphic to $P$. Furthermore, if $H \prec G$ for a connected $3$-graph $H$ of smaller order, then $P \prec H$ by the choice of $G$ and hence $P \prec G$, a contradiction. Thus, we can use Theorem~\ref{theo:inductive construction} to inductively construct infinitely many graphs belonging to $\ca H_3$.
	
	By Theorems~\ref{theo:coloring_P^M} and \ref{theo:characterisation H_r}, $ \ca T (r,1) \subset \ca H_r$. Note that the set $\ca T (r,1)$ is non-empty (see \cite{rizzi1999indecomposable}), and for $r\geq 4$, it does not contain any graph isomorphic to $P^{\ca M}$, where $\ca M$ is any multiset of $r-3$ perfect matchings of $P$. Hence, we can use Theorem~\ref{theo:inductive construction} to inductively construct infinitely many graphs belonging to $\ca H_r$.
\end{proof}

%\begin{cor}
%\label{cor:H_r infinite}
%If $r\geq 4$, then $\ca H_r$ is an infinite set.
%\end{cor}
%
%\begin{proof}
%By Theorems~\ref{theo:coloring_P^M} and \ref{theo:characterisation H_r}, $ \ca T (r,1) \subset \ca H_r$. Note that the set $\ca T (r,1)$ is non-empty (see \cite{rizzi1999indecomposable}), and for $r\geq 4$, it does not contain any graph isomorphic to $P^{\ca M}$, where $\ca M$ is any multiset of $r-3$ perfect matchings of $P$. Hence, we can use Theorem~\ref{theo:inductive construction} to inductively construct infinitely many graphs of class $2$ that cannot be colored by a connected $r$-graph of smaller order. By Theorem~\ref{theo:characterisation H_r}, all these graphs belong to $\ca H_r$.
%\end{proof}
%
%\begin{cor}
%\label{cor:H_3 = P or infinite}
%Either $\ca H_3 = \{P\}$ or $\ca H_3$ is an infinite set.
%\end{cor}
%
%\begin{proof}
%If $\ca H_3 \neq \{P\}$, then there is a smallest $3$-graph $G$ that cannot be colored by $P$. Note that $G$ is of class 2 and not isomorphic to $P$. Furthermore, if $H \prec G$ for a connected $3$-graph $H$ of smaller order, then $P \prec H$ by the choice of $G$ and hence $P \prec G$, a contradiction. Thus, as in the proof of Corollary~\ref{cor:H_r infinite}, we can use Theorem~\ref{theo:inductive construction} to inductively construct infinitely many graphs belonging to $\ca H_3$.
%\end{proof}

\subsection{Simple $r$-graphs}

In \cite{MTZ_r_graphs} the authors also asked whether for every $r\geq 4$, there is a connected $r$-graph coloring all simple $r$-graph. In this section we answer this question by showing that there is no finite set of connected $r$-graphs $\ca H_r'$ such that every connected simple $r$-graph can be colored by an element of $\ca H_r'$.

\begin{lem}[\cite{jin2017covers}]\label{lem:1-fator_avoiding_r-1_edges}
	Let $r$ be a positive integer, $G$ be an $r$-graph and $F \subseteq E(G)$. 
	If $|F|\le r-1$, then $G-F$ has a $1$-factor.
\end{lem}

Recall that, for an $r$-graph $G$ and an odd set $X\subseteq V(G)$, an edge-cut $\partial_G(X)$ is \emph{tight} if it consists of exactly $r$ edges. 

\begin{lem}\label{Proposition-tightcut-to-tightcut}
	Let $r\geq3$, let $G,H$ be connected $r$-graphs and let $f$ be an $H$-coloring of $G$. If $F\subseteq E(G)$ is a tight edge-cut in $G$, then $f(F)$ is a tight edge-cut in $H$.
\end{lem}

\begin{proof}
	Since $ F $ is a tight edge-cut, we have $|f(F)|\leq r.$
	Suppose that $|f(F)|<r.$ By Lemma \ref{lem:1-fator_avoiding_r-1_edges}, $H-f(F)$ has a perfect matching  $M$. Thus, $f^{-1}(M)$ is a perfect matching of $G$ such that $f^{-1}(M)\cap F = \emptyset$, a contradiction. Therefore, $|f(F)|=r$, and let $H_1,\dots, H_m$ be the components of $H-f(F)$.
	
	We first claim that the two endvertices of each edge in $ f(F)$ are in different components of $H-f(F)$. By contradiction, suppose that there is an edge $xy\in f(F)$ such that $x$ and $y$ are on the same component $H'$ of $H-f(F)$. Let $T$ be an $xy$-path contained in $H'$. Then, $f^{-1}(E(T)\cup \{xy\} )$ induces a $2$-regular subgraph in $G$ (see Observation \ref{obs:coloring_basics} $ (iii) $) and intersects $F$ exactly once, a contradiction. 
	
	The remaining proof is split into two cases as follows.
	
	{\bf Case 1.}\ $H-f(F)$ has a component of odd order.
	
	If $m>2$, then there is an odd component $H'$ with $\vert \partial_G(V(H'))| <r$, since $H-f(F)$ has at least two components of odd order, a contradiction. Hence, $H-f(F)$ has exactly two components, which are of odd order and therefore, $f(F)$ is a tight edge-cut in $H$.
	
	{\bf Case 2.}\ Every component of $H-f(F)$ is of even order. 
	
	Let $\tilde{H}$ be the graph obtained from $H$ by identifying all vertices in $V(H_i)$ to a new vertex for each $i \in \{1,\ldots,m\}$.
	Since every component is of even order, $\tilde{H}$ is an eulerian graph on $|f(F)|=r$ edges. 
	
	Now, we shall prove that $\tilde{H}$ is bipartite.
	Suppose by contradiction that $\tilde{H}$ has an odd circuit of length $2t+1$. This means that there is an odd number of components $H_{i_1},\dots,H_{i_{2t+1}}$ in $H-f(F)$ such that, for all $j\in \Z_{2t+1}$ there is an edge $x_jy_{j+1}\in f(F)$ such that $x_j\in V(H_{i_j})$ and $y_{j+1}\in V(H_{i_{j+1}})$. Moreover, for all $j\in\Z_{2t+1}$ there is an $x_jy_j$-path $T_j$ contained in the component $H_{i_j}$, i.e.\ such that $E(T_j)\cap f(F)=\emptyset$. Consider the circuit $C$ induced by    $x_jy_{j+1}$ and all edges of $T_j$ for all $j\in\Z_{2t+1}$. Then $|E(C)\cap f(F)|=2t+1$ and $f^{-1}(E(C))$ induces a $2$-regular subgraph in $G$ such that $|F\cap f^{-1}( E(C))|=2t+1$, a contradiction.	
	
	Since $\tilde{H}$ is a bipartite graph, we can assume without loss of generality that there is an $s\in\{1,\dots,m-1\}$ such that $f(F) = \partial_H(W)$, where $W = V(H_1)\cup\dots\cup V(H_s)$. Note that $|W|$ is even since every component of $H-f(F)$ has even order. Thus, a perfect matching $M$ of $H$ is such that $|M\cap \partial_H(W)| = |M\cap f(F)|$ is even. But then $|f^{-1}(M)\cap F|$ is even as well, a contradiction.
\end{proof}

\begin{lem}\label{Lemma-H-noclass1-subgraph}
	Let $r \geq 3$, let $ G $ and $ H $ be two $ r $-graphs, and let $ X $ be a subset of $ V(H) $ such that $ \partial_H(X) $ is a tight cut and $ \chi'(H/X^c)= r$. If  $H\prec G $,  then $H/X \prec G $.
\end{lem}

\begin{proof}
	Assume that $f$ is an $H$-coloring of $G$. 
	Label the edges of $\partial_H(X)$ as $ e_1,\ldots, e_r $.
	Since $ \chi'(H/X^c)= r$,  
	the subset  $ E(H[X])\cup \partial_H(X)$ of $ E(H) $ can be partitioned into $ r $ pairwise disjoint matchings, denoted by $ M_1,\ldots, M_r $, such that each edge of $ \partial_H(X) $ is contained in exactly one of them. Without loss of generality, we may assume $ e_i\in M_i $ for each $i\in\{1,\ldots,r\} $.  Note that $ E(G)= f^{-1}(E(H))= f^{-1}(E(H[X^c]))\cup f^{-1}(M_1)\cup\ldots\cup f^{-1}(M_r)$.
	Moreover, for convenience, every edge and every vertex  of $ H/X $ is
	labeled as in $H$. We define a mapping $f'\colon E(G)\to E(H/X)$ as follows. For every $e\in E(G)$, set  $$f'(e)=\begin{cases}
		f(e) & \text{ if } e\in  f^{-1}(E(H[X^c])); \\
		e_i & \text{ if } e\in f^{-1}(M_i), \text{ for } i\in\{1,\ldots,r\}.
	\end{cases}$$ 
	To conclude the proof, we shall show that  $f'$ is an $ H/X $-coloring of $ G $.  Let $ v$ be a vertex of $V(G) $. If $f(\partial_G(v)) =\partial_H(u)$ for some vertex $ u\in  X^c\subset V(H) $, then $f'(\partial_G(v)) =f(\partial_G(v)) =\partial_H(u)=\partial_{H/X}(u)$ by the definition of $ f' $. If $f(\partial_G(v)) =\partial_H(u)$ for some vertex $ u\in X $, then  the image  under $ f $ of each edge of $ \partial_G(v) $ is contained in one of $ M_1, \ldots, M_r $. Hence, the image  under $ f' $ of each edge of $ \partial_G(v) $ appears once in $ \partial_{H/X}(w_X) $. This implies $f'(\partial_G(v))=\partial_{H/X}(w_X)$. Thus, $f'$ is an $ H/X $-coloring of $ G $.
\end{proof}

For any graph $ G $, the number of isolated vertices of $ G $ is denoted by $ iso(G) $.
A simple graph $ H $ is \emph{regularizable} if we can obtain a regular graph from $H $ by replacing each edge of $ H $ by a nonempty set of parallel edges. We need the following lemma, which follows from two results of \cite{Berge-Vergnas-1978} and \cite{Pulleyblank-1979}. The equivalence of the first two statements is shown in \cite{Berge-Vergnas-1978}; the equivalence of the first and the third statement is shown in \cite{Pulleyblank-1979}.

\begin{lem}\label{Lemma-regular=factor}
	Let $ G $ be a simple connected graph which is not bipartite with two partition sets of the same cardinality. The following statements are equivalent:
	\begin{itemize}
		\item  $ iso(G-S) < |S|  $, for all $ S \subseteq V(G)$.
		\item   $ G $ is regularizable \cite{Berge-Vergnas-1978}.
		\item for every $ v\in V(G) $, both $ G-v $ and $ G $ have a $ \{K_{1,1}, C_m\colon m\geq3\} $-factor \cite{Pulleyblank-1979}.
	\end{itemize}
\end{lem}

\begin{lem}\label{Lemma-contract-H-class1}
	Let $r\geq 3$, let $ G$ and $ H $ be $ r $-graphs, where $H$ is connected, and let $S \subseteq V(G) $ such that $ \partial_G(S) $ is a tight cut and $ G[S] $ has no $ \{K_{1,1}, C_m:m\geq3\} $-factor. If $ G $ has an $ H $-coloring $f\colon E(G)\to E(H)$
	and $\partial_H(X)=f(\partial_G(S))$ for an $X \subseteq V(H)$, then   $H/X$ or $H/X^c$ is a bipartite graph with two partition sets of the same cardinality.
\end{lem}
\begin{proof}
	Suppose to the contrary that both $H/X$ and $H/X^c$ are not bipartite graphs with two partition sets of the same cardinality.   By Lemma \ref{Proposition-tightcut-to-tightcut},  the edge-cut $\partial_H(X)$ is tight and so both $H/X$ and $H/X^c$ are $ r $-regular. Thus, the underling graphs of $H/X$ and $H/X^c$ are both regularizable and hence,  both $H/X-w_X$ and $H/X^c-w_{X^c}$ have a $ \{K_{1,1}, C_m:m\geq3\} $-factor, by Lemma \ref{Lemma-regular=factor}. Let $ H' $ be the union of these two factors. Note that $ H' $ is a $ \{K_{1,1}, C_m:m\geq3\} $-factor of $ H $, which contains no edge of $ \partial_H(X) $. Since $\partial_H(X)=f(\partial_G(S))$ and by Observation  \ref{obs:coloring_basics} $ (iv) $, $ G $ has a $ \{K_{1,1}, C_m:m\geq3\} $-factor, which contains no edge of $ \partial_G(S) $. This is a contradiction to the assumption that $ G[S] $ has no $ \{K_{1,1}, C_m:m\geq3\} $-factor.
\end{proof}

	Let  $G$ be an $r$-regular graph with a vertex $v\in V(G)$. A \emph{Meredith extension} of $ G $ at $ v $ is the following operation. Delete the vertex $ v $ from $ G $ and add a copy $K$ of the complete bipartite graph $ K_{r,r-1} $. Finally add $ r $ edges between $ V(G-v)$ and $ V(K) $ such that the resulting graph is $ r $-regular.

\begin{lem}[Rizzi \cite{rizzi1999indecomposable}]\label{Lemma-keep-r-graph}
	Let $ G $ be a graph and $ X\subseteq V(G) $ with $ |X| $ odd. If $G/X$ and $G/X^c$ are both $ r $-graphs, then $ G $ is an $r$-graph.
\end{lem}

\begin{theo}\label{theo:reduction_simple_case}
	Let $r\geq 3$ and let $\ca H$ be a set of connected $r$-graphs such that every  $H\in \ca H$ does not contain a  non-trivial tight edge-cut $\partial_H(X)$ such that  $H/X$ or $H/X^c$ is class 1.
	If 
	every connected simple $r$-graph can be colored by an element of $\ca H$, 
then every connected $r$-graph can be colored by an element of $\ca H$.
\end{theo}

\begin{proof}
Let $G$ be an arbitrary $r$-graph. By applying a Meredith extension on every vertex of $ G $, we obtain a simple $ r $-regular graph $ G^e $. From the fact that both $ G $ and $ K_{r,r} $ are $ r $-graphs, we know that $ G^{e} $ is also an $ r $-graph by Lemma \ref{Lemma-keep-r-graph}. Hence, there is $H \in \ca H$ such that $H \prec G^{e} $. Let $f$ be an $ H $-coloring of $G^e$. Note that for any induced subgraph $ G' $ of $ G^e $ isomorphic to $ K_{r,r-1}$, the edge-cut $ \partial_{G^e }  (V(G'))$   is tight, and so $f(\partial_{G^e }  (V( G')))$ is also tight in $ H $ by Lemma \ref{Proposition-tightcut-to-tightcut}. Let $X \subset V(H)$ such that $\partial_{H}(X)=f(\partial_{G^e }  (V( G')))$. Since  $ K_{r,r-1} $  contains no $ \{K_{1,1}, C_m\colon m\geq3\} $-factor, Lemma \ref{Lemma-contract-H-class1} implies that  $H/X$ or $H/X^c$ is a bipartite graph with two partition sets of the same cardinality. In particular,   $H/X$ or $H/X^c$ is class 1, which implies that $X$ or $X^c$ is a single vertex by the choice of $ \ca H $. Therefore, the edge-cut $\partial_{G^e }  (V( G'))$  is mapped to a trivial edge-cut of $ H $ under $ f $. Since $ G' $ was chosen arbitrarily, we conclude that $G$ also has an $ H $-coloring, which completes the proof.
\end{proof}

%The following theorem is the main result of this subsection.
%
%\begin{theo}\label{Theorem-color-simple-r-graph}
%	Let $ r\geq4 $ be an integer. If there is a set $\ca H_r'$ of connected $r$-graphs such that for every simple $r$-graph $G$ there is a graph $H \in \ca H_r'$ with $H\prec G $, 
%	then $|\ca H_r'| \geq p'(r-3,6) + 1 \ge 2$. 
%\end{theo}
%
%\begin{proof}
%	By Lemma \ref{Lemma-H-noclass1-subgraph}, we can assume that for every $H \in \ca H'_r$, there is no non-trivial tight edge-cut $\partial_H(X)$ such that $H/X$ or $H/X^c$ is class $1$. Hence, by Theorem \ref{theo:reduction_simple_case}, every $r$-graph (parallel edges allowed) can be colored by an element of $\ca H_r'$. The statement then follows from Theorem \ref{thm: subset H}.
%\end{proof}

We obtain the main result of this section as a corollary.

\begin{cor}\label{cor:color-simple-r-graphs}
	Let $ r\geq3 $ and let $\ca H_r'$ be a set of connected $r$-graphs such that every connected simple $r$-graph can be colored by an element of $\ca H_r'$.
\begin{itemize}
%\item[$i)$] $\ca H_3' =\ca H_3$.'
\item[$i)$] If the Petersen Coloring Conjecture is false, then $\ca H_3'$ is an infinite set.
\item[$ii)$] If $r \geq 4$, then $\ca H_r'$ is an infinite set. 
\end{itemize}	
	 
\end{cor}

%\begin{cor}\label{cor:color-simple-r-graphs}
%	Let $ r\geq4 $ be an integer. There is no finite set $\ca H_r'$ of connected $r$-graphs such that every simple $r$-graph can be colored by an element of $\ca H_r'$. 
%\end{cor}

\begin{proof}
	By Lemma \ref{Lemma-H-noclass1-subgraph} we can contract suitable subsets of vertices of graphs in $\ca H_r'$ to obtain a set $\ca H_r''$ of connected $r$-graphs with the following properties.
\begin{itemize}
\item Every connected simple $r$-graph can be colored by an element of $\ca H_r''$.
\item For every $H \in \ca H''_r$, there is no non-trivial tight edge-cut $\partial_H(X)$ such that $H/X$ or $H/X^c$ is class $1$.
\end{itemize}	
Hence, by Theorem \ref{theo:reduction_simple_case}, every connected $r$-graph can be colored by an element of $\ca H_r''$. Thus, $\ca H_r \subset \ca H_r''$ and hence, $\ca H_r''$ is an infinite set by Corollary~\ref{cor:H_3 = P or infinite}. By the construction of $\ca H_r''$ we have $|\ca H_r'| \geq |\ca H_r''|$, and hence, $\ca H_r'$ is also an infinite set.
\end{proof}

\section{Concluding remarks} \label{Sec: final remarks}

\subsection{Quasi-ordered sets}
Jaeger \cite{Jaeger1980} initiated the study of the Petersen Coloring Conjecture in terms of partial ordered sets. DeVos, Ne\v{s}et\v{r}il and Raspaud \cite{DeVos_etal_2007} studied cycle-continuous mappings and asked whether there is an infinite set $ \ca G $ of  bridgeless graphs such that every two of them are cycle-continuous incomparable, i.e.\ there is no cycle-continuous map between any two graphs in $\ca G$.
\v{S}\'amal \cite{Robert_2017} gave an affirmative answer to the above question by constructing such an infinite set $ \ca G $ of  bridgeless cubic graphs.
In fact, he also mentioned that
this result can be considered in view of a quasi-order induced by cycle-continuous mappings on the set of bridgeless cubic graphs. That is, this quasi-ordered set contains infinite antichains.

For every integer $r \geq 1$, $H$-colorings of $r$-graphs induce a quasi-order on the set of $r$-graphs.
Then, our result on $ r $-graphs can be restated  as follows: for any $ r\geq4 $, there is an infinite set $ \ca H_r $ of  $ r $-graphs such that each of them  is  incomparable to any other $ r $-graph, and such infinite set exists for $ r=3 $ if the Petersen Coloring Conjecture is false.  
In particular, the set $\ca H_r $ is an infinite antichain.

\subsection{Open problems}

The edge connectivity of an $r$-graph is equal to $r$ or it is an even number.  
We have shown that $\ca T(r,r-2) \cup \ca T(r,1) \subseteq \ca H_r$.
Thus, for $r \not =5$, for each possible edge-connectivity $t$ there is 
a $t$-edge-connected $r$-graph in $\ca H_r$. For $r=5$, we do not know any 
$5$-edge-connected $5$-graph with this property, see \cite{MMSW_pdpm}
for a discussion of this topic. However, we know only a finite 
number of $t$-edge-connected $r$-graphs of $\ca H_r$ if $t \geq 3$. 

\begin{prob}
	For $r,t \geq 3$, does $\ca H_r$ contain infinitely many $t$-edge-connected 
	$r$-graphs?
\end{prob}

It is also not clear whether $\ca H_r$ contains elements
of $\ca T(r,k)$ for $k \in \{2, \dots, r-3\}$. 
So far, these sets are not determined for $k \in \{1, \dots, r-3\}$. 
Indeed, we even do not know the order of their elements. 
Let 
$o(r,k)$ be the order of the graphs of $\ca T(r,k)$. 

\begin{prob}\label{o(r,k)}
	For all $r \geq 3$ and $k \in \{1, \dots, r-2 \}\colon$ Determine $o(r,k)$.  
\end{prob}

By our results, $o(r,r-2) = 10$.
By results of Rizzi \cite{rizzi1999indecomposable}, $o(r,1) \leq 2 \times 5^{r-2}$. We conjecture the following to be true.

\begin{con} \label{conj: order}
	For all $r \geq 3$ and $k \in \{2, \dots, r-2\} \colon o(r,k-1) \geq o(r,k)$.
\end{con} 

If Conjecture \ref{conj: order} would be true, then it would follow
with Corollary~\ref{cor:coloring_graphs_in_S(r,k)} that $\ca T(r,k) \subset \ca H_r$ 
for each $k \in \{1, \dots, r-2\}$. 

Similar problems arise for simple $r$-graphs. Let $o_s(r,k)$ be the smallest
order of a simple $r$-graph $G$ with $\pi(G)=k$.
%\begin{prob}\label{o(r,k)}
%	For all $r \geq 3$ and $k \in \{1, \dots, r-2 \}\colon$ Determine $o_s(r,k)$.  
%\end{prob}
%If $r$ is odd and $\pi(G)=r-2$ and $G'$ is obtained from $G$ by replacing 
%a vertex with a complete graph $K_r$, then $\pi(G')=r-2$.
Small simple $r$-graphs of class 2 can be obtained as follows. Consider a perfect matching $M$ of $P$ and the graph $G=P+(r-3)M$. Let $H$ be a simple $r$-graph of smallest order and $v \in V(H)$. Then, $H$ is class 1 and $|V(H)|=r+1$ if $r$ is odd and $|V(H)|=r+2$ if $r$ is even. Now, replace appropriately five vertices of $G$ by $H-v$ to obtain a simple $r$-graph $G'$. Since $H$ is class 1 and $\pi(G)=r-2$, we have $\pi (G')=r-2$.   Therefore, if $r$ is odd, then $o_s(r,r-2) \leq 5(r+1)$ and
if $r$ is even, then $o_s(r,r-2) \leq 5(r+2)$. Furthermore, bounds for $o_s(r,k)$ can be obtained by using Meredith extensions, since if $G'$ is a Meredith extension of an $r$-graph $G$, then $\pi(G')=\pi(G)$.

%By applying these extensions to $P^{\ca M}$, where $\ca M$ consists of $r-3$
%copies of one perfect matching of $P$
%we obtain the following upper bounds for the order of
%smallest simple $r$-graphs of class 2.	Let $r \geq 4$ be an integer. If $r$ is odd, then $o_s(r,r-2) \leq 5(r+1)$.
%	If $r$ is even, then $o_s(r,r-2) \leq 10r$.

%It might also be true that $o_s(r,k-1) \geq o_s(r,k)$ for all 
%	$r \geq 3$ and $k \in \{2, \dots, r-2\}$.

\bibliography{Lit_reg_graphs}{}
\addcontentsline{toc}{section}{References}
\bibliographystyle{abbrv}

\end{document}